\documentclass[12pt,final]{article}   
\usepackage{amsmath}
\usepackage{amsthm}
\usepackage{amssymb}
\usepackage{amsfonts}
\usepackage{latexsym}               % LaTeX Symbol Font,amssymb
\usepackage{amsgen}
\usepackage[mathscr]{eucal}
\usepackage{mathrsfs}
\usepackage{bm}
\usepackage{enumerate}

\usepackage{mathptmx} % other option is txfonts
\usepackage{color}
\usepackage[dvips]{graphicx}

\usepackage{showkeys}

\newtheorem{lem}{Lemma}[section]
\newtheorem{thm}[lem]{Theorem}
\newtheorem{pro}[lem]{Proposition}

\newtheorem{rem}[lem]{Remark}

\newcommand{\bqn}{\begin{equation}}
\newcommand{\eqn}{\end{equation}}
\newcommand{\beqx}{\begin{equation*}}
\newcommand{\eeqx}{\end{equation*}}
\newcommand{\barr}{\begin{array}}
\newcommand{\earr}{\end{array}}
\newcommand{\beqn}{\begin{eqnarray}}
\newcommand{\eeqn}{\end{eqnarray}}
\newcommand{\beqnx}{\begin{eqnarray*}}
\newcommand{\eeqnx}{\end{eqnarray*}}
\newcommand{\bmt}{\begin{multline}}
\newcommand{\emt}{\end{multline}}

\numberwithin{equation}{section}

\newcommand{\D}{\partial}

\newcommand{\al}{\alpha}

\newcommand{\ga}{{\gamma}}

\newcommand{\de}{\delta}

\newcommand{\la}{\lambda}

\newcommand{\ro}{\rho}

\newcommand{\Om}{\Omega}

\newcommand{\tu}{\mbox{$\tilde u$}}
\newcommand{\tv}{\mbox{$\tilde v$}}
\newcommand{\tV}{\mbox{$\tilde V$}}

\newcommand{\tro}{\mbox{$\tilde \ro$}}
\newcommand{\tphi}{\mbox{$\tilde \phi$}}

\newcommand{\er}{\eqref}
\newcommand{\lb}{\label}
\newcommand{\qu}{\quad}

\begin{document}

%%%%%%%%%%%%%%%%%%% Title %%%%%%%%%%%%%%%%%%%%%

\title{Stability and existence of stationary solutions to \\
the Euler--Poisson equations \\
in a domain with a curved boundary}

\author{Masahiro Suzuki${}^1$ and Masahiro Takayama${}^2$}

\date{%
\normalsize
${}^1$%
Department of Computer Science and Engineering, 
Nagoya Institute of Technology,
\\
Gokiso-cho, Showa-ku, Nagoya, 466-8555, Japan
\\ [7pt]
${}^2$%
Department of Mathematics, 
Keio University, 
\\
Hiyoshi, Kohoku-ku, Yokohama, 223-8522, Japan
}

\maketitle

%%%%%%%%%%%%%%%%%% Abstract %%%%%%%%%%%%%%%%%%%%%%%%%%%%%%%
\begin{abstract}
The purpose of this paper is to mathematically investigate the formation of 
a plasma sheath near the surface of walls immersed in a plasma, 
and to analyze qualitative information of such a sheath layer.
In the case of planar wall, Bohm proposed a criterion 
on the velocity of the positive ion for the formation of sheath,
and several works gave its mathematical validation.
It is of more interest to analyze the criterion for the nonplanar wall.
In this paper, we study the existence and asymptotic stability of stationary solutions 
for the Euler-Poisson equations in a domain of which boundary is drawn by a graph.
The existence and stability theorems are shown by assuming that
the velocity of the positive ion satisfies the Bohm criterion at infinite distance.
What most interests us in these theorems is that
the criterion together with a suitable necessary condition
guarantees the formation of sheaths 
as long as the shape of walls is drawn by a graph.
\end{abstract}

\begin{description}
\item[{\it Keywords:}]
plasma;
sheath;
Bohm criterion;
initial--boundary value problem;
long-time behavior;
convergence rate.

\item[{\it 2010 Mathematics Subject Classification:}]
82D10; %Plasma
35M12; %Boundary value problems for equations of mixed type
35M13; %Initial-boundary value problems for equations of mixed type
35A01; %Existence problems: global existence, local existence, non-existence
35B35. %Stability
\end{description}

%%%%%%%%%%%%%%%%% TEXT START %%%%%%%%%%%%%%%%%%%%%%%%%

\newpage

\section{Introduction}\lb{S1}

The purpose of this paper is to mathematically investigate the formation of 
a plasma boundary layer, called as a \emph{sheath}, 
near the surface of materials immersed in a plasma, 
and to analyze qualitative information of such a layer.
The sheath appears when a material is surrounded
by a plasma and the plasma contacts with its surface.
Because the thermal velocities of electrons are much higher than those of ions, 
more electrons tend to hit the material compared with ions.
This makes the material negatively charged with respect to the surrounding plasma. 
Then the material with a negative potential attracts and accelerates ions toward the surface, 
while repelling electrons away from it. 
Eventually, there appears a non-neutral potential region near the surface, 
where a nontrivial equilibrium of the densities is achieved. 
This non-neutral region is referred as to the sheath. 
This layer shields the plasma from the negatively charged body, 
and the thickness is the same order of the Debye length.
For more details of  physicality of the sheath development, 
we refer the reader to \cite{D.B.1, F.C.1,I.L.1, LL, K.R.1, K.R.2}.

For the formation of sheath, Langmuir in \cite{I.L.1} observed that 
positive ions must enter the sheath region with a sufficiently large kinetic energy. 
Bohm in \cite{D.B.1} proposed the original \emph{Bohm criterion} for the plasma 
containing electrons and only one component of mono-valence ions, 
which states that the ion velocity at the plasma edge must exceed the ion acoustic speed,
in the case of planar wall.
Mathematically, the planar wall cases have been investigated 
by using the Euler--Poisson equations \er{eq1}--\er{eq3} below.
Ambroso, M\'ehats, and Raviart did a pioneering work \cite{AMR} where
the unique existence of monotone stationary solutions was proved
over a bounded interval, provided that the Bohm criterion holds.
Furthermore, Ambroso \cite{A.A.} numerically checked that
solutions of initial--boundary value problems
approach the stationary solutions constructed in \cite{AMR}
as the time variable becomes large.
Suzuki \cite{M.S.1} derived a necessary and sufficient condition, including the Bohm criterion,
for the unique existence of monotone stationary solutions over a half space.
Furthermore, he showed the asymptotic stability of stationary solutions 
by assuming a condition slightly stronger than the criterion.
After that, the stability theorem was shown under the Bohm criterion in \cite{NOS}.
For a multicomponent plasma containing electrons and several components of ions,
similar results to \cite{NOS,M.S.1} were obtained in \cite{M.S.2} 
under the generalized Bohm criterion derived by Riemann in \cite{K.R.2}.
These results validated mathematically the Bohm criterion
and defined the fact that the sheath corresponds to the stationary solution.
Let us also mention the results on the quasi-neutral limit problem
as letting the Debye length in the Euler-Poisson equations tend to zero.
G\'erard-Varet, Han-Kwan, and Rousset in \cite{GHR1,GHR2}
studied the problems over the half space with various boundary conditions.
In particular, the result in \cite{GHR2} clarified the fact that
the thickness of the boundary layer is of order of the Debye length.
We also introduce a couple of results studying the equations over the whole space.
The time-global solvability and quasi-neutral limit problem
were investigated in \cite{GP1} and \cite{CG1}, respectively.
The traveling wave solutions were established in \cite{CDMS2}.

For the planer wall cases, 
the formation of the sheath has been well-understood as above.
We are now interested in the cases that walls are nonplanar.
For this direction, Jung, Kwon, and Suzuki in \cite{JKS1} studied
the existence and quasi-neutral limit of stationary solutions over an annulus.
They focused only on spherical symmetry solutions
and then proposed {\it a Bohm criterion for the annulus},
which essentially differs from the original Bohm criterion.
It is of interest to know how the Bohm criterion depends on the shape of walls.
The main purpose of this paper is 
to analyze the sheath formation for more general domains.
In this situation, the plasma no longer flows unidirectionally,
although the above results studied only unidirectional flows.
We remark that few mathematical studies have been reported on
steady states having multidirectional flows for compressible fluids.

After a suitable nondimensionalization, 
the Euler-Poisson equations is written by
\begin{subequations}\label{eq0}
   \begin{gather}
   \ro_t+\nabla \cdot(\ro \bm{u})=0,
   \label{eq1}\\
   \bm{u}_t +  \left( \bm{u} \cdot \nabla \right) \bm{u}
    + K \nabla (\log \ro) = \nabla \phi,
   \label{eq2}\\
   \Delta \phi=\ro-e^{-\phi},
   \label{eq3}
   \end{gather}
where unknown functions
$\rho$, $\bm{u}=(u_1,u_2,u_3)$, and $-\phi$
represent the density and velocity of the positive ions and
the electrostatic potential, respectively.
Furthermore, $K$ is a positive constant.
The first equation is the conservation of mass, and the second one
is the equation of momentum in which the pressure gradient 
and electrostatic potential gradient as well as the convection effect are taken into account.
The third equation is the Poisson equation,
which governs the relation between the potential and the density of charged particles.
It is obtained by assuming the Boltzmann relation in which the electron density is 
given by $\ro_e=e^{-\phi}$.
We study an initial--boundary value problem of \er{eq0}
in a domain 
\[
\Omega:=\{x=(x_1,x_2,x_3)\in \mathbb R^3 \, | 
\, x_1>M(x_2,x_3)\}
\quad \text{for $M \in \cap_{k=1}^\infty H^k(\mathbb R^2)$}.
\]
The initial and boundary data are prescribed as
 \begin{gather}
  (\rho,\bm{u})(0,x)=(\rho_0,\bm{u}_0)(x), %\qu
  %\inf_{x\in\Omega}\ro_{0}(x)>0, \qu
  %\inf_{x\in\Omega}(u_0^2-\gamma K \ro_0)(x)>0,
  %\qu
  %\sup_{x\in\Omega}u_0 \cdot \nu (x)<0,
  \lb{ini1}
  \\
\lim_{x_1\to\infty}(\ro,u_1,u_2,u_3,\phi)(t,x_1,x_2,x_3)=(1,u_+,0,0,0),
 \label{bc1}
 \\
  \phi(t,M(x_2,x_3),x_2,x_3)=\phi_b \quad \text{for} \
  (x_2,x_3)\in \mathbb R^2,
  \label{bc2}
 \end{gather}
 \end{subequations}
where $u_+<0$ and $\phi_b\in \mathbb R$ are constants.
The unit outer normal vector of the boundary 
$\D\Om=\{x\in \mathbb R^3 \, | \, x_1=M(x_2,x_3)\}$ is represented by 
\begin{equation*}
\bm{n}(x_2,x_3)=(n_1,n_2,n_3)(x_2,x_3)
:=\left(\frac{-1}{\sqrt{1+|\nabla M|^2}},
\frac{\partial_{x_2}M}{\sqrt{1+|\nabla M|^2}},
\frac{\partial_{x_3}M}{\sqrt{1+|\nabla M|^2}} \right)(x_2,x_3).
\end{equation*}

We construct solutions in the region, where the following two conditions hold:
  \begin{gather}
  \inf_{x\in \Om}\ro(x)>0,
  \label{po1}
  \\
  \inf_{x \in \D\Om}
  \frac{\bm{u}(x)\cdot\nabla(M(x_2,x_3)-x_1)}{\sqrt{1+|\nabla M(x_2,x_3)|^2}}-\sqrt{K}>0,
  \lb{super1'}
  \end{gather}
by assuming the same conditions for the initial data $(\rho_0,\bm{u}_0)$:
\[
 \inf_{x\in \Om}\ro_0(x)>0, \quad 
  \inf_{x \in \D\Om}
  \frac{\bm{u}_0(x)\cdot\nabla(M(x_2,x_3)-x_1)}{\sqrt{1+|\nabla M(x_2,x_3)|^2}}-\sqrt{K}>0. 
\]
In particular, the supersonic outflow condition \er{super1'} is necessary for 
the well-posedness of this initial--boundary value problem,
because it guarantees that no boundary condition is suitable for equations \er{eq1} and \er{eq2}.
In this setting, we do not need any compability conditions.
For the end state of velocity $u_+$, 
we assume the Bohm criterion and the supersonic outflow condition:
\begin{gather}
u_+^2>K+1, \quad u_+<0,
\lb{Bohm1}\\
\inf_{x \in \D\Om}
\frac{-u_+}{\sqrt{1+|\nabla M(x_2,x_3)|^2}}-\sqrt{K}>0.
\lb{asp1'}
\end{gather}
We remark that \er{asp1'} is required
if solutions to problem \er{eq0} is established
in a neighborhood of the constant state 
$(\rho,u_1,u_2,u_3,\phi)=(1,u_+,0,0,0)$,
which is a trivial solution for the case $\phi_b=0$.
%Indeed, owing to \er{asp1'}, 
%we put no boundary condition for the Euler equation.

We study the existence and stability of stationary solutions
over the domain $\Om$ with the curved boundary.
The main difficulty lies on the fact that the stationary problem
is still given by a boundary value problem to a hyperbolic--elliptic system,
although the problem over a half space or an annulus can be reduced to 
a system of ordinary differential equations.
It is also worth to pointing out that the hyperbolic equations of 
the stationary problem over $\Om$ do not have any initial data and boundary conditions.
As this point, our problem differs from standard situations.
In addition, we do not assume the smallness of
the function $M$ representing the boundary $\D \Om$.

Let us discuss more details of the difficulty mentioned above and the strategies to resolve.
For the situation solving hyperbolic equations without initial and boundary data,
one may first think of the application of theorems in \cite{Fr1}, 
which discuss the solvability for the linear case, and then
linearize the Euler--Poisson equations so that
the hyperbolic and elliptic parts are decoupled.
However, the inductive scheme to solve the nonlinear problem does not work well for our situation.
Generally speaking, this scheme works for time-evolution problem 
by taking the time variable small enough.
We cannot find any alternative quantity to the time variable in the steady case.
For the same reason, the contraction mapping principle is also not useful.
Therefore, we must solve the stationary problem with a totally different approach.

Our approach is that we first show the time-global solvability of problem \er{eq0} and 
then construct stationary solutions by using the global solutions.
These procedures are in the reverse order to standard ways
in which a stationary solution is first constructed and then 
the time-global solvability is shown in the neighborhood of the stationary solution
by combining time-local solvability and an a priori estimate.
Let us explain the idea to have time-global solutions for unknown steady states.
For example, one can have a priori estimates of solutions of
%an initial--boundary value problem of 
some inhomogeneous parabolic equations over bounded domains
even if the long-time behavior of solutions is not anticipated (see \cite{KS1}).
The key of the proof is the dissipative structure which makes
solutions of the corresponding homogeneous equations decay
exponentially fast as time tends to infinity.
On the other hand, the stability theorems in \cite{NOS,M.S.1,M.S.2} imply that 
the solution $(\rho,u_1,u_2,u_3,\phi)$ to problem \er{eq0} with $\phi_b=0$
converges the constant state $(1,u_+,0,0,0)$ %in a weighted Sobolev space 
exponentially fast as time tends to infinity.
For the case $\phi_b\neq 0$, after suitable reformulation, 
all effects coming from $\phi_b\neq 0$ are represented by 
inhomogeneous terms in the equations.
Therefore, the dissipative structure enables us to obtain the a priori estimate
of solutions to our problem.
For the construction of stationary solutions, 
we define a sequence by the time-global solution sifted the time variable $t$
to $t+kT^*$ for any $T^*>0$ and $k \in \mathbb N$, and then show that 
this sequence converges a time-periodic solution with a period $T^*$ as $k$ tends to infinity. 
By using the arbitrary of period, it can be concluded that
this time-periodic solution is independent of $t$.
%For more details, see Section \ref{S5}.
Before closing this section, we give our notation used throughout this paper.

%\medskip
%
%\noindent
%{\bf Notation.}
%For a real number $p$, $[p]$ is defined by a maximum integer not greater than $p$.
The notation $\langle u,v \rangle$ means the inner product of $u,v \in \mathbb R^4$.
We use $c$ and $C$ to denote generic positive constants.
Let us also denote a generic positive constant depending additionally
on other parameters $\alpha$, $\beta$, $\ldots$
by $C[\alpha,\, \beta,\, \ldots]$. 
For a nonnegative integer $k$, ${\cal {B}}^k (\Sigma)$ stands for
the space of functions whose derivatives up to  
$k$-th order are continuous and bounded over $\Sigma$.
Furthermore, ${\cal {B}}^\infty (\Sigma)$ is defined by
$\cap_{k=1}^\infty {\cal {B}}^k (\Sigma)$.
For $ 1 \leq p \leq \infty$ and a nonnegative integer $k$, 
$L^p(\Omega)$ is the Lebesgue space;
$W^{k,p}(\Omega)$ is the $k$-th order Sobolev space in the $L^p$ sense;
$H^k(\Omega)$ is the $k$-th order Sobolev space in the $L^2$ sense, 
equipped with the norm $\Vert\cdot\Vert_k$.
We note $H^0 = L^2$, 
$\Vert\cdot\Vert:=\Vert\cdot\Vert_0$, and $H^\infty:=\cap_{k=1}^\infty H^k$.
We also define weighted Sobolev spaces $H^k_{\al}(\Om)$ and $H^k_{\al,\la} (\Om)$
for $\al > 0$ and $\la \geq 2$ by
\begin{gather*}
H^k_{\al} (\Om):=\left\{ f \in H^k(\Om) \,\left| \, 
\|f\|_{k,\al}^2=\sum_{j=0}^k\int_{\Om} e^{\al x_1} |\nabla^j f|^2\, dx < \infty \right\}\right.,
\\
H^k_{\al,\la} (\Om):=\left\{ f \in H^k(\Om) \,\left| \, 
\|f\|_{k,\al,\la}^2=\sum_{j=0}^k\int_{\Om} 
w_{\alpha,\la}|\nabla^j f|^2\, dx < \infty \right\}\right.,
\end{gather*}
where
\[
w_{\al,\la}(x_1)
:=\left(1+\min\left\{\al,(1+|M|_{L\infty(\mathbb R^2)})^{-1}\right\} x_1 \right)^\la.
\]
Note that there exist $c$ and $C$ independent of $\alpha$ such that
\begin{equation}\lb{eqiv0}
c\|f\|_{k,\al}\leq\|e^{\al x_1/2}f\|_k\leq C\|f\|_{k,\al}
\quad \text{for $f \in H^k_{\al} (\Om)$ and $\alpha \in (0,1]$}.
\end{equation}
The notation $C^k([0,T];{\mathcal H})$ means 
the space of $k$-times continuously
differentiable functions on the interval $[0,T]$ with values in 
some Hilbert space ${\mathcal H}$.

\section{Main results}\lb{S3/2}
Before mentioning our main results, we introduce a result in \cite{M.S.1}
which showed the unique existence of stationary solutions
over a one-dimensional half space $\mathbb R_+:=\{x_1>0\}$.
Stationary solutions $(\tro,\tu,\tphi)(x_1)$ solve the system
\begin{subequations}\lb{sp0}
\begin{gather}
   (\tro\tu)'=0,
   \lb{sp1}\\
   {\tu}\tu'+K(\log \tro)'={\tphi}',
   \lb{sp2}\\
   \tphi''=\tro-e^{-\tilde{\phi}}
   \lb{sp3}
   \end{gather}
with the conditions
\begin{gather}
  \inf_{x_1\in\mathbb R_{+}}\tro(x_1)>0, \qu
  \lim_{x_1\rightarrow \infty}(\tro,\tu,\tphi)(x_1)=(1,u_+,0), \qu
  \tphi(0)=\phi_b.
  \lb{sp4}
\end{gather}
\end{subequations}
Under the Bohm criterion \er{Bohm1},
the unique existence of stationary solutions $(\tro,\tu,\tphi)$ was established
as in the following lemma.
\begin{lem}[\cite{M.S.1}]\lb{1.1}
Let $u_+$ satisfy \er{Bohm1}. There exist a constant $\delta>0$ such that
if $|\phi_b| < \delta$, then problem \eqref{sp0}
has a unique monotone solution $(\tilde{\rho} ,\tilde{u}, \tilde{\phi}) \in 
{\cal B}^{\infty}(\overline{\mathbb{R}_{+}})$. Moreover, it satisfies
\begin{equation}\label{ses1}
|\partial_{x_1}^j(\tilde{\rho}-1)|
+|\partial_{x_1}^j(\tilde{u}-u_{+})|
+|\partial_{x_1}^j\tilde{\phi}|
 \leq C|\phi_b| e^{-\al x}
\quad \text{for} \quad 
j=0,1,2,\cdots, 
\end{equation}
where $\al<1$ and $C$ are positive constants independent of $\phi_b$.
\end{lem}

From now on we discuss our main results.
We first show the unique existence of stationary solutions 
$(\ro^s,\bm{u}^s,\phi^s)=(\ro^s,u^s_1,u^s_2,u^s_3,\phi^s)$
over the domain $\Om$ with the curved boundary by regarding 
$(\ro^s,u^s_1,u^s_2,u^s_3,\phi^s)(x)$ 
as a perturbation of $(\tro,\tu,0,0,\tphi)(\tilde{M}(x))$, where
\begin{equation}\lb{tM1}
\tilde{M}(x):=x_1-M(x_2,x_3). 
\end{equation}
The stationary solutions satisfy the equations
\begin{subequations}\label{seq0}
   \begin{gather}
   \nabla \cdot(\ro^s \bm{u}^s)=0,
   \label{seq1}\\
    \left( \bm{u}^s \cdot \nabla \right) \bm{u}^s
    + K \nabla (\log \ro^s) = \nabla \phi^s,
   \label{seq2}\\
   \Delta \phi^s=\ro^s-e^{-\phi^s}
   \label{seq3}
    \end{gather}
and the conditions
   \begin{gather}
  \inf_{x\in \Om}\ro^s(x)>0,
  \label{spo1}
  \\
  \lim_{x_1\to\infty}(\ro^s,u^s_1,u^s_2,u^s_3,\phi)(t,x_1,x_2,x_3)=(1,u_+,0,0,0),
  \label{sbc1}
  \\
  \phi^s(t,M(x_2,x_3),x_2,x_3)=\phi_b \quad \text{for} \
  (x_2,x_3)\in \mathbb R^2.
  \label{sbc2}
 \end{gather}
 \end{subequations}
The existence result is summarized in the following theorem.
It is worth to pointing out that we do not require 
any smallness assumptions for the function $M$ 
representing the boundary of the domain $\Om$. %as
%$\D \Om =\{x\in \mathbb R^3 \, | \, x_1=M(x_2,x_3)\}$.

\begin{thm}\lb{1.2}
Let $m \geq 3$ and $u_+$ satisfy \er{Bohm1} and \er{asp1'}.
There exist positive constants $\beta \leq \al/2$,
where $\al$ is defined in Lemma \ref{1.1},
and $\de$ such that if $|\phi_b| \leq \de$, 
then stationary problem \er{seq0} 
has a unique solution $(\rho^s,\bm{u}^s,\phi^s)$ as
\begin{gather*}
(\rho^s,u^s_1,u^s_2,u^s_3,\phi^s)
-(\tro\circ\tilde{M},\tu\circ\tilde{M},0,0,\tphi\circ\tilde{M})
\in [H^m_\beta(\Om)]^4\times H^{m+1}_\beta(\Om),
\\
\|(\rho^s-\tro\circ\tilde{M},u^s_1-\tu\circ\tilde{M},u^s_2,u^s_3)\|_{m,\beta}^2
+\|\phi^s-\tphi\circ\tilde{M}\|_{m+1,\beta}^2 \leq C |\phi_b|,
\end{gather*}
where $C$ is a positive constant independent of $\phi_b$.
\end{thm}

We also show the stability of stationary solutions
in both exponential and algebraic weighted Sobolev spaces.
The papers \cite{M.S.1,M.S.2} pointed out that
system \er{eq1}--\er{eq3} itself does not have
the dissipative effect in the usual function space,
however there appear those effects in the weighted space.
Therefore, we employ the weighted space.
In addition, we remark that
the smallness of $M$ is not assumed in the exponential weight case.
\begin{thm}\lb{1.3}
Let $u_+$ satisfy \er{Bohm1} and \er{asp1'}.
There exist positive constants $\beta \leq \al/2$,
where $\al$ is defined in Lemma \ref{1.1},
and $\de$ %depending on $\beta$ 
such that if $\|(\rho_0-\rho^s,\bm{u}_0-\bm{u}^s)\|_{3,\beta}+|\phi_b| \leq \de$, 
then initial--boundary value problem \er{eq0} has a unique time-global solution
$(\rho,\bm{u},\phi)$ with \er{po1} and \er{super1'} in the following space.
\[
(\rho-\rho^s,\bm{u}-\bm{u}^s,\phi-\phi^s) \in
\left[\bigcap_{i=0}^1 C^i([0,T];H^{3-i}_\beta(\Om))\right]^4
\times C([0,T];H^{5}_\beta(\Om)).
\]
Moreover, it holds that
\begin{equation*}%\lb{decay1}
\sup_{x \in \Om}|(\rho-\rho^s,\bm{u}-\bm{u}^s,\phi-\phi^s)(t,x)| \leq Ce^{-\ga t}
\quad \text{for $t \in [0,\infty)$},
\end{equation*}
where $C$ and $\ga$ are positive constants independent of $\phi_b$ and $t$.
\end{thm}

\begin{thm}\lb{1.4}
Let $\la \geq 2$, $\nu \in (0,\la]$, and $u_+$ satisfy \er{Bohm1} and \er{asp1'}.
There exist positive constants $\beta_0 \leq \beta$,
where $\beta$ is defined in Theorem \ref{2.1},
and $\de$ %depending on $\beta_0$ 
such that if $\|M\|_{5}+\|(\rho_0-\rho^s,\bm{u}_0-\bm{u}^s)\|_{3,\beta_0,\la}+|\phi_b| \leq \de$, 
then initial--boundary value problem \er{eq0} has a unique time-global solution
$(\rho,\bm{u},\phi)$ with \er{po1} and \er{super1'} in the following space.
\[
(\rho-\rho^s,\bm{u}-\bm{u}^s,\phi-\phi^s) \in
\left[\bigcap_{i=0}^1C^i([0,T];H^{3-i}_{\beta_0,\la}(\Om))\right]^4
\times C([0,T];H^{5-i}_{\beta_0,\la}(\Om)).
\]
Moreover, it holds that
\begin{equation*}%\lb{decay2}
\sup_{x \in \Om}|(\rho-\rho^s,\bm{u}-\bm{u}^s,\phi-\phi^s)(t,x)| \leq C(1+t)^{-\la+\nu}
\quad \text{for $t \in [0,\infty)$},
\end{equation*}
where $C$ is a positive constant independent of $\phi_b$ and $t$.
\end{thm}

In this paper, we focus only on the discussion on Theorems \ref{1.2} and \ref{1.3},
because Theorem \ref{1.4} can be shown by the essentially same method
as in \cite{NOS} which proved the stability of stationary solutions 
to problem \er{eq0} with $M=0$.
An outline of the proof of Theorem \ref{1.4} will be discussed in Appendix \ref{S6}.
Now we mention some remarks from a physical point of view.

\begin{rem}
Bohm originally derived criterion \er{Bohm1} 
for the formation of sheaths only in the planer wall case.
What most interests us in Theorems \ref{1.2} and \ref{1.3} is that
his criterion with the supersonic outflow condition \er{asp1'} 
also guarantees the formation of sheaths in any case 
that the shape of walls is drawn by a graph.
We emphasize again that \er{asp1'} is a necessary condition
for the well-posedness of problem \er{eq0}.
\end{rem}

This paper is organized as follows.
In Section \ref{S2}, we start from rewriting 
initial--boundary value problem \er{eq0}
by introducing a perturbation from the stationary solution
%$(\tro,\tu,0,0,\tphi)(\tilde{M}(x))$ 
over the half space.
Section \ref{S4} is devoted to showing the time-global solvability of the rewritten problem
in the exponential weighted Sobolev space.
%The proof is based on the energy method in the weighted Sobolev spaces.
We construct stationary solutions in Section \ref{S5} 
by using the time-global solutions established above.
The stability of stationary solutions is also shown in the same weighted space.
%To this end, we define a sequence by the time-global solution sifted the time variable $t$
%to $t+kT^*$ for any $T^*>0$, and then show that 
%this sequence converges a time-periodic solution as $k$ tends to infinity. 
Appendixes \ref{S6} and \ref{Appendix1} provide the proofs of 
the stability in the algebraic weighted Sobolev space and general inequalities, respectively.

\section{Reformulation}\lb{S2}
For mathematical convenience, we begin by reformulating
initial--boundary value problem \er{eq0}.
Let us introduce new functions 
\[
\tilde{V}(x_1)={}^{t}(\tv,\tilde{\bm{u}})(x_1)
={}^{t}(\tv,\tu_1,\tu_2,\tu_3)(x_1):={}^{t}(\sqrt{K}\log\tro,\tu,0,0)(x_1), \quad
%\tilde{V}(x):={}^t(\tv,\tu_1,\tu_2,\tu_3)(\tilde{M}(x)), \ \
v(t,x):=\sqrt{K}\log\rho(t,x)
\]
and perturbations
\begin{align*}
\Psi(t,x)&={}^t(\psi,\bm{\eta})(t,x)={}^t(\psi,\eta_1,\eta_2,\eta_3)(t,x)
:={}^t(v,u_1,u_2,u_3)(t,x)-\tilde{V}(\tilde{M}(x)),
\\
\sigma(t,x)&:=\phi(t,x)-\tphi(\tilde{M}(x)),
\end{align*}
where $\tilde{M}(x)$ is defined in \er{tM1}.
Then, from \er{eq0} and \er{sp0}, we have the reformulated problem for $(\Psi,\sigma)$:
\begin{subequations}\lb{re0}
\begin{gather}
\D_{t}\Psi+
\sum_{j=1}^3A^j[\tV+\Psi]\D_{x_j}\Psi
=\begin{bmatrix}
0 \\ \nabla \sigma
\end{bmatrix}
+B[\tV',\nabla M]\Psi
+\begin{bmatrix}
0 \\ \bm{h}[\tV,\tV',\nabla M]
\end{bmatrix},
\lb{re1}
\\
\Delta \sigma-\sigma=K^{-1/2}{\psi}+g_0[\psi,\tv]
+g_1[\sigma,\tphi]+g_2[\tphi',\nabla M],
\lb{re2}
\\
\lim_{|x|\to\infty}(\Psi,\sigma)(x)=0, 
\lb{re4}
\\
\sigma(M(x_2,x_3),x_2,x_3)=0,
\lb{re5}
\\
\Psi(0,x)=\Psi_0(x):=
{}^t(\sqrt{K}\log\rho_0,\bm{u}_0)(x)-\tilde{V}(\tilde{M}(x)).
\lb{re3}
\end{gather}
\end{subequations}
Here the $4 \times 4$ symmetric matrices $A^j$, $4 \times 4$ matrix $B$,
and $3 \times 1$ matrix $\bm{h}$ are defined as
\begin{align*}
A^1[\tV+\Psi]&:=
\begin{bmatrix}
(\tu_1+\eta_1) & \sqrt{K} & 0 & 0
\\
\sqrt{K} & (\tu_1+\eta_1) & 0 & 0
\\
0 & 0 & (\tu_1+\eta_1) & 0
\\
0 & 0 & 0 & (\tu_1+\eta_1)
\end{bmatrix},
\\
A^2[\tV+\Psi]&:=
\begin{bmatrix}
(\tu_2+\eta_2) & 0 & \sqrt{K} & 0
\\
0 & (\tu_2+\eta_2) & 0 & 0
\\
\sqrt{K} & 0 & (\tu_2+\eta_2) & 0
\\
0 & 0 & 0 & (\tu_2+\eta_2)
\end{bmatrix},
\\
A^3[\tV+\Psi]&:=
\begin{bmatrix}
(\tu_3+\eta_3) & 0 & 0 & \sqrt{K}
\\
0 & (\tu_3+\eta_3) & 0 & 0
\\
0 & 0 & (\tu_3+\eta_3) & 0
\\
\sqrt{K} & 0 & 0 & (\tu_3+\eta_3)
\end{bmatrix},
\\
B[\tV',\nabla M]&:=
\begin{bmatrix}
0 & -\tv' & \tv'\D_{x_2} M & \tv'\D_{x_3} M
\\
0 & -\tu' & \tu'\D_{x_2} M & \tu'\D_{x_3} M
\\
0 & 0 & 0 & 0
\\
0 & 0 & 0 & 0
\end{bmatrix},
\quad
\bm{h}[\tV,\tV',\nabla M]:=
\begin{bmatrix}
0
\\
-\tu\tu' \D_{x_2} M
\\
-\tu\tu' \D_{x_3} M
\end{bmatrix}.
\end{align*}
The scalar values $g_0$, $g_1$, and $g_2$ are defined as
\begin{align*}
g_0[\psi,\tv]
&:=K^{-1/2}{\psi}(e^{\tv/\sqrt{K}}-1)
+e^{\tv/\sqrt{K}}\left(e^{\psi/\sqrt{K}}-1-K^{-1/2}{\psi}\right),
%\\
%&\ =K^{-1/2}{\psi}(e^{\tv/\sqrt{K}}-1)
%+e^{\tv/\sqrt{K}}\left(\frac{1}{2}e^{\theta\psi/\sqrt{K}}\right)\frac{\psi^2}{K},
\\
g_1[\sigma,\tphi]
&:=(e^{-\tilde{\phi}}-1)\sigma-e^{-\tilde{\phi}}(e^{-\sigma}-1+\sigma),
%\\
%&\ =(e^{-\tilde{\phi}}-1)\sigma+e^{-\tilde{\phi}}
%\left(\frac{1}{2}e^{-\theta \sigma}\right)\sigma^2,
\\
g_2[\tphi',\nabla M]
&:=\sum_{i=2}^3 (-\tphi''( \D_{x_i}M)^2+\tphi' \D_{x_ix_i}^2M).
\end{align*}
It is straightforward to check that \er{super1'} is equivalent to 
\begin{gather}
\inf_{x \in \D\Om, \, \Phi \in \mathbb R^4, \, |\Phi|=1 } 
\left\langle \sum_{j=1}^3 n_j(x_2,x_3) A^j[\Psi(t,x)+\tV(\tilde{M}(x))]\Phi, 
\Phi \right \rangle>0.
\lb{super1}
\end{gather}

%The advantage of using \er{re1} and \er{re2} is 
%that these equations have the same form 
%as those equations used in the paper \cite{NOS}, 
%which show the stability of the stationary solution over the half space,
%if $\bm{h}$ and $g_2$ are zero.
We remark that it suffices to show Theorems \ref{2.1} and \ref{2.2} below 
for the completion of the proof of Theorems \ref{1.2} and \ref{1.3}, respectively.
\begin{thm}\lb{2.1}
Let $m \geq 3$ and $u_+$ satisfy \er{Bohm1} and \er{asp1'}.
There exist positive constants $\beta \leq \al/2$,
where $\al$ is defined in Theorem \ref{1.1},
and $\de$ such that if $|\phi_b| \leq \de$, 
then the stationary problem corresponding to \er{re0} 
has a solution $(\Psi^s,\sigma^s) \in [H^m_\beta(\Om)]^4\times H^{m+1}_\beta(\Om)$ 
satisfying \er{super1} and
\begin{equation}\lb{ses2}
\|\Psi^s\|_{m,\beta}^2 +\|\sigma^s\|_{m+1,\beta}^2 \leq C |\phi_b|,
\end{equation}
where $C$ is a positive constant independent of $\phi_b$.
\end{thm}
\begin{thm}\lb{2.2}
Let $u_+$ satisfy \er{Bohm1} and \er{asp1'}.
There exist positive constants $\beta \leq \al/2$,
where $\al$ is defined in Theorem \ref{1.1},
and $\de$ %depending on $\beta$ 
such that if $\|\Psi_0\|_{3,\beta}+|\phi_b| \leq \de$, 
then initial--boundary value problem \er{re0} has a unique time-global solution
$(\Psi,\sigma) \in [\bigcap_{i=0}^1 C^i([0,T];H^{3-i}_\beta(\Om))]^4
\times C([0,T];H^{5}_\beta(\Om))$ 
with \er{super1}. Moreover, it holds that
\begin{equation}\lb{decay1}
\sup_{x \in \Om}|(\Psi-\Psi^s,\sigma-\sigma^s)(t,x)| \leq Ce^{-\ga t}
\quad \text{for $t \in [0,\infty)$,}
\end{equation}
where $C$ and $\ga$ are positive constants independent of $\phi_b$ and $t$.
\end{thm}

\section{Time-global solvability}\lb{S4}
This section deals with the time-global solvability of 
initial--boundary value problem \er{re0} 
for small initial data $\Psi_0$ and boundary data $\phi_b$.
We notice that inhomogeneous terms $\bm{h}$ in \er{re1} and $g_2$ in \er{re2} vanish if $\phi_b=0$.
In this case, the essentially same proof as in \cite{NOS,M.S.1} works, 
and one can see that $(\Psi,\sigma)$ exists globally in time and decays exponentially fast 
in the exponential weighted Sobolev space as $t$ tends to infinity.
Even for the case $\phi_b\neq0$, this dissipative structure enables us to prove that 
the $H^m_\beta$-norm of solutions is bounded 
by those of initial data $\Psi_0$ and inhomogeneous terms $\bm{h}$ and $g_2$.
We often use this kind of technique in 
studying parabolic equations over bounded domains (for instance, see \cite{KS1}).
The next theorem provides the unique existence of 
time-global solutions to problem \er{re0}.

\begin{thm}\lb{4.1}
Let $m \geq 3$ and $u_+$ satisfy \er{Bohm1} and \er{asp1'}.
There exist positive constants $\beta \leq \al/2$,
where $\al$ is defined in Theorem \ref{1.1},
and $\de$ depending on $\beta$ such that if 
$\|\Psi_0\|_{m,\beta}+|\phi_b| \leq \de$, 
then initial--boundary value problem \er{re0} has a unique time-global solution
$(\Psi,\sigma) \in [\bigcap_{i=0}^1 C^i([0,\infty);H^{m-i}_\beta(\Om))]^4
\times C([0,\infty);H^{m+2}_\beta(\Om))$
with \er{super1}. Moreover, it holds that
\begin{equation}\lb{apes1}
\sup_{t\in[0,\infty)}\left(
\|\Psi(t)\|_{m,\beta}^2+\|\D_t \Psi(t)\|_{m-1,\beta}^2
+\|\sigma(t)\|_{m+2,\beta}^2\right)
\leq C(\|\Psi_0\|_{m,\beta}^2+|\phi_b|),
\end{equation}
where $C$ is a positive constant depending on $\beta$ but
independent of $\phi_b$.
\end{thm}

The time-global solution $(\Phi,\sigma)$ with \er{apes1} 
can be constructed by a standard continuation argument 
using the time-local solvability in Lemma \ref{3.1} and 
the a priori estimate in Proposition \ref{4.2} below.
Here we use notation
\[
 N_{m,\al}(T):=\sup_{t\in [0,T]}\|\Psi(t)\|_{m,\al}.
\]
\begin{lem}\lb{3.1}
Suppose that $\Psi_0$ satisfies \er{super1} and 
belongs $H^m_{\al/2}(\Om)$ 
for $m \geq 3$ and $\alpha>0$ being in Theorem \ref{1.1}.
Let $\beta$ be a positive constant less than
$\al/2$ and $2 e^{m_*/2}$, where
\begin{gather*}
\lb{beta1}
\qu
m_*:=\min \left\{
\inf_{x \in \mathbb R_+} \left(-\tphi(x)\right), \
K^{-1/2}\inf_{x \in \Om}(v_0+\tv)-1
\right\}.
\end{gather*}
Then there exist positive constants $\delta$ and $T$ such that 
if $|\phi_b| < \delta$, problem \er{re0} has a unique solution 
$(\Psi,\sigma) \in [\bigcap_{i=0}^1 C^i([0,T];H^{m-i}_\beta(\Om))]^4
\times C([0,T];H^{m+2}_\beta(\Om))$ with \er{super1}.
\end{lem}
%\begin{proof}
%This can be proved in the same way as Lemma 3.1 in \cite{M.S.1}.
%\end{proof}

\begin{pro}\label{4.2}
Let $m \geq 3$ and $u_+$ satisfy \er{Bohm1} and \er{asp1'}.
Suppose that $(\Psi,\sigma) \in [\bigcap_{i=0}^1 C^i([0,T]$ $;H^{m-i}_\beta(\Om))]^4
\times C([0,T];H^{m+2}_\beta(\Om))$
be a solution to problem \er{re0} 
with \er{super1}.
There exist positive constants $\beta \leq \al/2$,
where $\al$ is defined in Theorem \ref{1.1},
and $\de$ depending on $\beta$
such that if $N_{m,\beta}(T)+|\phi_b| < \delta$,
the following estimate holds.
\begin{equation}\lb{apes2}
\sup_{t\in[0,T]}\left(
\|\Psi(t)\|_{m,\beta}^2+\|\D_t \Psi(t)\|_{m-1,\beta}^2
+\|\sigma(t)\|_{m+2,\beta}^2\right)
\leq C(\|\Psi_0\|_{m,\beta}^2+|\phi_b|),
\end{equation}
where $C$ is a positive constant depending on $\beta$
but independent of $\phi_b$.
\end{pro}

Since Lemma \ref{3.1} can be proved in much the same way as Lemma 3.1 in \cite{M.S.1},
we prove only Proposition \ref{4.2} in the remainder of this section.
In subsection \ref{S4.1}, we derive estimates of $\sigma$ 
solving elliptic equation \er{re2}.
Subsections \ref{S4.1} and \ref{S4.2} deal with 
basic and higher-order estimates of $\Psi$
solving hyperbolic equations \er{re1}, respectively.
The a priori estimate is completed in subsection \ref{S4.3}.

\subsection{Elliptic estimates}\lb{S4.1}
This subsection provides
\footnote{We remark that all constants $C$ in subsection \ref{S4.1} are independent of $\beta$.}{estimates}
of $\sigma$ solving elliptic equation \er{re2}. 
Let us first show the lower and upper bounds.
\begin{lem}\lb{ell0}
Under the same assumption as in Proposition \ref{4.2}, it holds that
\begin{gather}
\sup_{x\in\Om}(\sigma+\tphi)(t,x) \leq M_1,
\quad M_1:=\max\left\{\sup_{x\in\Om}|\tphi(\tilde{M}(x))|, \
-\inf_{x\in\Omega}\frac{\tv}{\sqrt{K}}+1\right\},
\lb{bounds1}\\
\inf_{x\in\Om}(\sigma+\tphi)(t,x) \geq -M_2,
\quad M_2:=\max\left\{\sup_{x\in\Om}|\tphi(\tilde{M}(x))|, \
\sup_{x\in\Om}\frac{\tv}{\sqrt{K}}+1\right\},
\lb{bounds2}\\
\sup_{x\in\Om}|\sigma(t,x)| \leq C (N_{m,\beta}(T)+|\phi_b|),
\lb{ellineq0}
\end{gather}
where $C$ is a positive constant independent of $\beta$, $\phi_b$, and $t$.
\end{lem}
\begin{proof}
For the proof of \er{bounds1}, %since the lower bound can be shown similarly.
let us set $\Phi:= (\tphi+\sigma) - M_1$.
It is straightforward to check from \er{sp3} and \er{re2} that
\begin{equation*}
\Delta \Phi =-e^{-\Phi-M_1}+ e^{(\psi+\tv)/\sqrt{K}}.
\end{equation*}
Multiply this by $\Phi^+:= \max\{\Phi, 0\}$, integrate it over $\Om$,
and use $\Phi^+(t,M(x_2,x_3),x_2,x_3)=0$ and $\lim_{|x|\to\infty}\Phi^+(t,x)=0$.
Then using $\Phi^+ \geq 0$ and letting $N_{m,\beta}(T)$ be small enough,  we have
\begin{align*}
\int_{\Om} |\nabla\Phi^+|^2 dx
=\int_{\Om} \left\{e^{-\Phi^+-M_1}-e^{(\psi+\tv)/\sqrt{K}} \right\}\Phi^+ dx
\leq \int_{\Om}\left\{e^{-M_1} -e^{\inf(\tv/\sqrt{K})-1}\right\}\Phi^+ dx
\leq 0,
\end{align*}
which gives \er{bounds1}. Similarly, one can have the bound \er{bounds2}.

Next we show \er{ellineq0}. It is first seen from \er{ses1}, \er{bounds1}, and \er{bounds2} that
\begin{gather}
 |g_0| \leq C(|\psi|+|\phi_b|)|\psi|,
 \lb{esg0}\\
 |g_1| \leq C(|\sigma|+|\phi_b|)|\sigma|,
 \lb{esg1}\\
 |g_2| \leq C|\phi_b| e^{-\al( x_1-M(x_2,x_3))} |\nabla M|.
 \lb{esg2}
\end{gather}
Multiply \er{re2} by $\sigma$, integrate it by parts over $\Om$,
and use \er{re5} to get
\begin{align*}
\int_\Om |\nabla \sigma|^2 dx
+\int_\Om \left(e^{-\sigma-\tphi}-e^{-\tphi}\right)\sigma dx
&=\int_\Om \left(K^{-1/2}{\psi}+g_0+g_2\right)\sigma dx
\\
&\leq \mu \|\sigma\|^2
+C[\mu](N_{m,\beta}^2(T)+|\phi_b|^2),
\end{align*}
where $\mu$ is a positive constant to be determined later and 
we have used \er{esg0}--\er{esg2}, Schwarz's inequality, and $M \in H^\infty(\Om)$
in deriving the last inequality.
On the other hand, by \er{ses1}, \er{bounds1}, and the mean value theorem,
the second term on the left hand side is estimated from below as
\[
\int_\Om\left(e^{-\sigma-\tphi}-e^{-\tphi}\right)\sigma dx 
\geq c \|\sigma\|^2.
\]
These two inequalities with sufficiently small $\mu>0$ lead to
\begin{equation}\lb{bounds3}
\|\sigma\|_1^2 \leq C(N_{m,\beta}^2(T)+|\phi_b|^2).
\end{equation}
Then applying Lemma \ref{Gell1} in Appendix \ref{Appendix1} to \er{re2}
and using \er{bounds1}, \er{bounds2}, \er{esg0}--\er{bounds3}, 
and $M \in H^\infty(\Om)$, we have
\begin{align*}
\|\sigma\|_2 \leq C(\|\psi\|+\|g_0\|+\|g_1\|+\|g_2\|)
%\\
%&\leq C(\|\psi(t)\|+\|\sigma(t)\|+|\phi_b|)
%\quad (\because \er{esg0},\er{esg2},\|g_2\|\leq C|\phi_b|)
%\\
\leq C(N_{m,\beta}(T)+|\phi_b|),
%\quad (\because \|\psi(t)\|\leq CN_{m,\beta}(T), \ \er{bounds3})
\end{align*}
which together with Sobolev's inequality yields \er{ellineq0}.
\end{proof}

From now on we estimate the $H^k_\beta$-norm of $\sigma$
by that of $\psi$ and the boundary data $\phi_b$.
\begin{lem}\lb{ell1}
Under the same assumption as in Proposition \ref{4.2}, it holds that
\begin{gather}
\|\sigma(t)\|_{1,\beta}^2 
\leq \{K^{-1}+D\beta^2+C(N_{m,\beta}(T)+|\phi_b|)\} \|\psi(t)\|_{0,\beta}^2+C|\phi_b|,
\lb{ellineq1}\\
\|\sigma(t)\|_{l+2,\beta}^2 
\leq C(\|\psi(t)\|_{l,\beta}^2+|\phi_b|^2) \quad \text{for $l=0,1,\ldots,m$,}
\lb{ellineq4}
\end{gather}
where $C$ and $D$ are positive constants independent of $\beta$, $\phi_b$, and $t$.
\end{lem}
\begin{proof}
Let us first show \er{ellineq1}. We see from \er{ses1}, \er{ellineq0}--\er{esg2},
$\beta \leq \al/2$, Sobolev's inequality, and $M \in H^\infty$ that
\begin{gather}
\|g_0(t)\|_{0,\beta} \leq C (N_{m,\beta}(T)+|\phi_b|)\|\psi(t)\|_{0,\beta}, 
%\quad \text{($\because$ \er{esg0})}
\lb{g0}\\
\|g_1(t)\|_{0,\beta} \leq C (N_{m,\beta}(T)+|\phi_b|)\|\sigma(t)\|_{0,\beta}, 
%\quad \text{($\because$ \er{ellineq0}, \er{esg1})}
\lb{g1}\\
\sup_{t\in [0,T]} \|g_2(t)\|_{m,\beta} \leq C |\phi_b|.
%\quad \text{($\because$ \er{ses1}, $\beta \leq \al/2$, $M \in H^\infty$)}
\lb{g2}
\end{gather}
Multiply \er{re2} by $e^{\beta x_1} \sigma$, integrate it by parts over $\Om$, 
and use \er{g0}--\er{g2} and Schwarz's inequality to get
\begin{align*}
{}&
\int_\Om e^{\beta x_1} |\nabla \sigma|^2
+\left(1-\frac{\beta^2}{2}\right)e^{\beta x_1}|\sigma|^2 dx
\\
&=\int_\Om e^{\beta x_1} \left(K^{-1/2}{\psi}+g_0+g_1+g_2\right)\sigma dx
%\\
%&\leq \frac{1}{4} \|\sigma\|_{0,\beta}^2
%+C(\|\psi\|_{0,\beta}^2+\|g_0\|_{0,\beta}^2+\|g_1\|_{0,\beta}^2+\|g_2\|_{0,\beta}^2)
%\quad \text{($\because$ Schwarz)}
\\
&\leq \frac{1}{4} \|\sigma\|_{0,\beta}^2
+C\{\|\psi\|_{0,\beta}^2+(N_{m,\beta}(T)+|\phi_b|)^2\|\sigma\|_{0,\beta}^2+|\phi_b|^2\}.
%\quad \text{($\because$ \er{g0}, \er{g1}, \er{g2})}
\end{align*}
Owing to $\beta \leq \al/2 \leq 1/2$, 
letting $N_{m,\beta}(T)+|\phi_b|$ be sufficiently small,
we have 
\begin{equation}\lb{ellineq3}
\|\sigma(t)\|_{1,\beta}^2 \leq C(\|\psi(t)\|_{0,\beta}^2+|\phi_b|^2).
\end{equation}

Once again, multiply \er{re2} by $e^{\beta x_1} \sigma$,
integrate it by parts over $\Om$, 
and estimate the result in a different way as above by using \er{ellineq3}.
\begin{align*}
{}&
\int_\Om e^{\beta x_1} |\nabla \sigma|^2+e^{\beta x_1}|\sigma|^2 dx
\\
&=
\int_\Om e^{\beta x_1}K^{-1/2}\psi\sigma dx
+\int_\Om \frac{\beta^2}{2}e^{\beta x_1}|\sigma|^2+e^{\beta x_1} (g_0+g_1+g_2)\sigma dx
\\
%&\leq \frac{1}{2} \|\sigma\|_{0,\beta}^2+\frac{1}{2K} \|\psi\|_{0,\beta}^2
%+\int_\Om \frac{\beta^2}{2}e^{\beta x_1}|\sigma|^2+e^{\beta x_1} (g_0+g_1+g_2)\sigma dx
%\quad \text{($\because$ Schwarz)}
%\\
%&\leq \frac{1}{2} \|\sigma\|_{0,\beta}^2+\frac{1}{2K} \|\psi\|_{0,\beta}^2
%+\int_\Om \frac{\beta^2}{2}e^{\beta x_1}|\sigma|^2 dx
%+(\|g_0\|_{0,\beta}+\|g_1\|_{0,\beta})\|\sigma\|_{0,\beta}
%+(|\phi_b|^{-1/2}\|g_2\|_{0,\beta})(|\phi_b|^{1/2}\|\sigma\|_{0,\beta})
%\quad \text{($\because$ Schwarz)}
%\\
&\leq \frac{1}{2} \|\sigma\|_{0,\beta}^2+\frac{1}{2K} \|\psi\|_{0,\beta}^2
+C(\beta^2+N_{m,\beta}(T)+|\phi_b|)\|\psi\|_{0,\beta}^2+C|\phi_b|.
%\quad \text{($\because$ \er{g0}--\er{ellineq3})}
\end{align*}
This immediately gives \er{ellineq1}.

We treat \er{ellineq4} for the case $l=0$. 
Multiplying \er{re2} by $e^{\beta x_1/2}$ yields
\begin{equation}\lb{elleq1}
\Delta (e^{\beta x_1/2} \sigma)-e^{\beta x_1/2} \sigma
=\beta\D_{x_1}(e^{\beta x_1/2} \sigma)
-\frac{\beta^2}{4}e^{\beta x_1/2} \sigma
+e^{\beta x_1/2} \left(K^{-1/2}{\psi}+g_0+g_1+g_2\right). 
\end{equation}
Applying Lemma \ref{Gell1} in Appendix \ref{Appendix1} to \er{elleq1} gives 
\begin{align}
\|e^{\beta x_1/2} \sigma\|_{2} 
&\leq C(\|e^{\beta x_1/2}\sigma\|_{1}
+\|e^{\beta x_1/2} \psi\|_{0}
+\|g_0\|_{0,\beta}
+\|g_1\|_{0,\beta}
+\|g_2\|_{0,\beta})
\notag\\
%&\leq C\|\sigma\|_{1,\beta}
%+C\|\psi\|_{0,\beta}
%+C|\phi_b|
%\quad (\because \text{$\|e^{\beta x_1/2} \cdot \|_{l} \leq C \|\cdot \|_{l,\beta}$,
%\er{g0},\er{g1},\er{g2}})
%\notag\\
&\leq C\|\psi\|_{0,\beta}+C|\phi_b|,
%&\quad (\because \text{\er{ellineq3}})
\lb{elleq4}
\end{align}
where we have also used \er{eqiv0} and \er{g0}--\er{ellineq3}
in deriving the last inequality.
Hence, \er{ellineq4} holds for $l=0$.

Next let us treat the case $l\geq 1$ by induction on $l$. 
By assuming \er{ellineq4} holds for $l=i$, we show \er{ellineq4} with $l=i+1$.
It is straightforward to see from \er{A1} and \er{A3} in Appendix \ref{Appendix1} that 
\begin{gather}
\|e^{\beta x_1/2}g_0\|_{i}\leq C\|\psi\|_{i,\beta},
%\lb{elleq2}\\
\quad 
\|e^{\beta x_1/2}g_1\|_{i}\leq C\|\sigma\|_{i,\beta}. 
\lb{elleq3}
\end{gather}
Applying Lemma \ref{Gell1} to \er{elleq1} again 
and using the induction hypothesis, we have
\begin{align*}
\|e^{\beta x_1/2} \sigma\|_{i+2} 
\leq C\|e^{\beta x_1/2}\sigma\|_{i+1}
+C\|e^{\beta x_1/2} \psi\|_{i}
+C\|e^{\beta x_1/2}g_2\|_{i}
%\\
%&\leq C\|\sigma\|_{i+1,\beta}
%+C\|\psi\|_{i,\beta}
%+C|\phi_b|
%\quad (\because \text{$\|e^{\beta x_1/2} \cdot \|_{l} \leq C \|\cdot \|_{l,\beta}$,\er{g2}})
%\\
\leq C\|\psi\|_{i,\beta}+C|\phi_b|.
%\quad (\because \text{assumptions of induction})
\end{align*}
This together with \er{eqiv0} leads to \er{ellineq4} with $l=i+1$.
Hence, we deduce \er{ellineq4} for all $l=0,1,2,\ldots,m$.
%Let us check \er{elleq3}. 
%The term $g_1$ can be written in the form 
%$g_1=(e^{-\tphi}-1)\sigma+e^{-\tphi}A(\sigma)\sigma$,
%where $A(0)=0$ and $A\in {\cal B}^\infty(B(0,1))$.
%Then we see from \er{A1} and \er{A3} in Appendix \ref{Appendix1} that 
%\begin{align*}
%\|e^{\beta x_1/2}g_1\|_i 
%&\leq C\|g_1\|_{i,\beta}
%\\
%&\leq C \|\sigma\|_{i,\beta}
%+C\left(|A(\sigma)|_{\infty}\|\sigma\|_{i,\beta}
%+\|A(\sigma)\|_{i}|\sigma|_{\infty,\beta}\right)
%\quad (\because \text{\er{ses1},\er{A1}})
%\\
%&\leq C \|\sigma\|_{i,\beta}+C\|A(\sigma)\|_{i}
%\quad (\because \text{\er{ellineq0},\er{elleq4}})
%\\
%&\leq C \|\sigma\|_{i,\beta}
%\quad (\because \text{\er{A3}})
%\end{align*}
%Similarly, \er{elleq2} can be shown.
%Hence, the proof is complete.
\end{proof}

\subsection{Basic estimate}\lb{S4.2}
This subsection is devoted to deriving an estimate of $L^2$-norm
of $\Psi$ solving hyperbolic equations \er{re1}.
Only in this subsection, we must to be careful to check the dependence of $\beta$
in order to take it suitably small.

For the derivation, we begin by introducing several equalities.
Taking the inner product of \er{re1} with the vector $2\Psi$, we have
\begin{gather}
\D_t\left(|\Psi|^2\right)
+\sum_{j=1}^3\D_{x_j}\left(\langle A^j[\tV+\Psi]\Psi,\Psi \rangle
-2\sigma\eta_{j}\right)
=-2\sigma(\nabla\cdot\bm{\eta})+\mathcal{R}_1,
\lb{bes1}\\
\mathcal{R}_1:=
\sum_{j=1}^3 \langle \{\D_{x_j}(A^j[\tV+\Psi])\}\Psi,\Psi \rangle
+2\langle B\Psi,\Psi\rangle+2\bm{h}\cdot\bm{\eta}.
\notag
\end{gather}
Furthermore, one can check from \er{re1} that 
${}^t(\nabla \cdot \hm{\eta},\nabla\psi)$ satisfies a system of equations:
\begin{align*}
{}&
\D_{t}
\begin{bmatrix}
\nabla \cdot \hm{\eta} \\ \nabla\psi
\end{bmatrix}
+\sum_{j=1}^3A^j[\tV+\Psi]\D_{x_j}
\begin{bmatrix}
\nabla \cdot \hm{\eta} \\ \nabla\psi
\end{bmatrix}
\notag\\
&=\begin{bmatrix}
\Delta \sigma \\ \bm{0}
\end{bmatrix}
+\begin{bmatrix}
\D_{x_1}(B\Psi)_{2}
+\nabla\cdot\bm{h}
-\sum_{i,j=1}^3 \{\D_{x_i}(\tu_j+\eta_j)\} \D_{x_j}\eta_{i} 
\\ \nabla(B\Psi)_1-\nabla(\tilde{\bm{u}}+\bm{\eta})\nabla \psi
%\lb{bes2}
\end{bmatrix},
\end{align*}
where $(B\Psi)_{l}$ means the $l$-th components of $B\Psi$.
Taking the inner product of this with ${}^t(2\nabla \cdot \hm{\eta},2\nabla\psi)$ leads to
\begin{gather}
\D_t\{(\nabla \cdot \bm{\eta})^2+|\nabla\psi|^2\}
+\sum_{j=1}^3 \D_{x_j}\left\{\!\left\langle\!
A^j[\tV+\Psi]
\begin{bmatrix}
\nabla \cdot \hm{\eta}
\\
\nabla \psi
\end{bmatrix},
\begin{bmatrix}
\nabla \cdot \hm{\eta}
\\
\nabla \psi
\end{bmatrix}
\!\right\rangle\!\right\}
=2(\Delta \sigma)(\nabla \cdot \bm{\eta})+{\mathcal R}_2,
\lb{bes3}\\
\begin{aligned}
{\mathcal R}_2:=&
\sum_{j=1}^3 \left\langle\!
\{\D_{x_j}(A^j[\tV+\Psi])\}
\begin{bmatrix}
\nabla \cdot \hm{\eta}
\\
\nabla \psi
\end{bmatrix},
\begin{bmatrix}
\nabla \cdot \hm{\eta}
\\
\nabla \psi
\end{bmatrix}
\!\right\rangle
+2\{\nabla(B\Psi)_1-\nabla(\tilde{\bm{u}}+\bm{\eta})\nabla \psi\}\cdot\nabla\psi
\\
&+2\left\{\D_{x_1}(B\Psi)_{2}
+\nabla \cdot \bm{h}
-\sum_{i,j=1}^3\{\D_{x_i}(\tu_j+\eta_j)\} \D_{x_j}\eta_{i}\right\}
(\nabla \cdot \bm{\eta}).
\end{aligned}
\notag
\end{gather}
To handle the terms having $\sigma$ 
on the right hand sides of \er{bes1} and \er{bes3},
we multiply \er{re2} by $2\nabla\cdot\bm{\eta}$ and rewrite the result as
\begin{align}
%{}&
2(\Delta \sigma-\sigma)(\nabla\cdot\bm{\eta})
%\notag\\
&=2K^{-1/2}\psi(\nabla\cdot\bm{\eta})
+2(g_0+g_1+g_2)(\nabla\cdot\bm{\eta})
%\quad (\because \text{\er{re2}})
\notag\\
&=2K^{-1}\psi\left(-\D_t\psi
-\sum_{j=1}^3 (\eta_j+\tu_j)\D_{x_j}\psi
+(B\Psi)_1\right)
+2(g_0+g_1+g_2)(\nabla\cdot\bm{\eta})
%\quad (\because \text{\er{re1}})
\notag\\
&=-K^{-1}\D_t(\psi^2)
-K^{-1}\sum_{j=1}^3\D_{x_j}\left\{(\eta_j+\tu_j)\psi^2 \right\}
+{\mathcal R}_3,
\lb{bes4}
\end{align}

\vspace{-4mm}

\[
{\mathcal R}_3:=K^{-1}\sum_{j=1}^3\left\{\D_{x_j}(\eta_j+\tu_j) \right\} \psi^2
+2K^{-1}\psi(B\Psi)_1+2(g_0+g_1+g_2)(\nabla\cdot\bm{\eta}),
\]
where we have also used the first component of \er{re1}
in deriving the second equality.

Furthermore, it is seen from \er{ses1}, \er{ellineq0}--\er{esg2}, 
Sobolev's and Schwarz's inequalities,
and $M \in H^\infty(\Om)$ that
\begin{gather}
|(\mathcal{R}_1,\mathcal{R}_2,\mathcal{R}_3)| 
\leq C(N_{m,\beta}(T)+|\phi_b|)|(\Psi,\nabla\Psi,\nabla\phi)|^2
+C|(\bm{h},g_2,\nabla\bm{h})||(\Psi,\nabla\Psi)|,
\lb{bes5}\\
|(\bm{h},g_2,\nabla\bm{h})| \leq C|\phi_b| e^{-\al( x_1-M(x_2,x_3))} |\nabla M|,
\lb{h1}
\end{gather}
where $C$ is a positive constant independent of $\beta$, $\phi_b$, and $t$.
From now on we estimate the $L^2$-norm of $\Psi$.

\begin{lem}\label{4.3}
Under the same assumption as in Proposition \ref{4.2},
it holds that 
\begin{equation}\lb{basic1}
\sup_{t\in[0,T]}\|\Psi(t)\|_{0,\beta}^2
\leq C\|\Psi_0\|_{1,\beta}^2
+\frac{C}{\beta}(N_{m,\beta}(T)+|\phi_b|)\sup_{t\in[0,T]}\|\nabla \Psi(t)\|_{0,\beta}^2
+\frac{C}{\beta}|\phi_b|,
\end{equation}
where $C$ is a positive constant independent of $\beta$, $\phi_b$, and $t$.
\end{lem}
\begin{proof}
Sum up \er{bes1}--\er{bes4},
multiply the result by $e^{\beta x_1}$, 
integrate it over $\Om$, 
and use Gauss's divergence theorem with \er{re4} 
to obtain
\begin{align}
{}&
\frac{d}{dt} \int_\Om e^{\beta x_1}
\left(|\Psi|^2+K^{-1}\psi^2
+(\nabla \cdot \bm{\eta})^2
+|\nabla \psi|^2\right)\,dx
\notag\\
&
\quad+\sum_{j=1}^3\int_{\D\Om} e^{\beta M(x_2,x_3)}\left(
\langle n_j A^j[\tV+\Psi]\Psi,\Psi \rangle 
+\left\langle\!
n_jA^j[\tV+\Psi]
\begin{bmatrix}
\nabla \cdot \hm{\eta}
\\
\nabla \psi
\end{bmatrix},
\begin{bmatrix}
\nabla \cdot \hm{\eta}
\\
\nabla \psi
\end{bmatrix}
\!\right\rangle\right)\,ds
\notag\\
&
\quad+\sum_{j=1}^3\int_{\D\Om} e^{\beta M(x_2,x_3)}n_j K^{-1}(\eta_j+\tu_j)\psi^2\,ds
\notag\\
&
\quad-\beta\int_{\Om} e^{\beta x_1}
(\langle A^1[\tV+\Psi]\Psi,\Psi \rangle
+K^{-1}(\eta_1+\tu_1)\psi^2
-2\sigma\eta_1)\,dx
\notag\\
&
\quad-\beta\int_{\Om} e^{\beta x_1}
\left\langle\!
A^1[\tV+\Psi]
\begin{bmatrix}
\nabla \cdot \hm{\eta}
\\
\nabla \psi
\end{bmatrix},
\begin{bmatrix}
\nabla \cdot \hm{\eta}
\\
\nabla \psi
\end{bmatrix}
\!\right\rangle\,dx
\notag\\
&=\int_{\Om} e^{\beta x_1}\left({\mathcal R}_1+{\mathcal R}_2+{\mathcal R}_3\right)\,dx
\notag\\
&\leq C(N_{m,\beta}(T)+|\phi_b|)\|\Psi\|_{1,\beta}^2+C|\phi_b|,
\lb{bes10}
\end{align}
where we have also used \er{ellineq4}, \er{bes5}, \er{h1}, 
$\beta\leq\alpha/2$, $M\in H^\infty(\Om)$, and Schwarz's inequality
in deriving the last inequality.

Let us estimate each terms on the left hand side from below separately.
The second term is nonnegative thanks to \er{super1}.
It can be shown by using \er{ses1} and $n_1u_+>0$ 
that the third term is also nonnegative as
\begin{align*}
{}&
\sum_{j=1}^3\int_{\D\Om} e^{\beta M(x_2,x_3)}n_j K^{-1}(\eta_j+\tu_j)\psi^2\,ds
\\
&\geq \int_{\D\Om} e^{\beta M(x_2,x_3)} n_1 K^{-1} u_+ \psi^2\,ds
-C(N_{m,\beta}(T)+|\phi_b|)\|\psi\|_{L^2(\D \Om)}^2
\\
& \geq 0.
\end{align*}
The last inequality follows from taking  $N_{m,\beta}(T)$ and $|\phi_b|$ small enough.
By using \er{ses1} and \er{ellineq1}, we estimate the fourth term as
\begin{align*}
{}&
-\beta\int_{\Om} e^{\beta x_1}
(\langle A^1[\tV+\Psi]\Psi,\Psi \rangle+K^{-1}(\eta_1+\tu_1)\psi^2
-2\sigma\eta_1)\,dx
\\
&\geq -\beta\int_{\Om} e^{\beta x_1}\{
(K^{-1}+1)u_+\psi^2+2\sqrt{K}\psi\eta_1+u_+|\bm{\eta}|^2-2\sigma\eta_1\}\,dx
\\
& \qquad
-C(N_{m,\beta}(T)+|\phi_b|)\|\Psi\|_{0,\beta}^2
\\
%&\geq  -\beta\int_{\Om} e^{\beta x_1}\{
%(K^{-1}+1)u_+\psi^2+2\sqrt{K}\psi\eta_1+u_+|\eta|^2\}\,dx
%-2\beta K^{-1/2}\|\psi\|_{0,\beta}\|\eta_1\|_{0,\beta}
%\\
%& \qquad -\{2\sqrt{D}\beta^2+\mu+C[\mu](N_{m,\beta}(T)+|\phi_b|)\}\|\Psi\|_{0,\beta}^2
%-C[\mu]|\phi_b|
%\\
& \geq \beta{\mathcal D}
-\{2\sqrt{D}\beta^2+\mu+C[\mu](N_{m,\beta}(T)+|\phi_b|)\}\|\Psi\|_{0,\beta}^2
-C[\mu]|\phi_b|,
\end{align*}
where $\mu$ is a positive constant to be determined later and
\[
{\mathcal D}:= -\int_{\Om} e^{\beta x_1}\{
(K^{-1}+1)u_+\psi^2+2\sqrt{K}\psi\eta_1+u_+|\bm{\eta}|^2\}\,dx
-2 K^{-1/2}\|\psi\|_{0,\beta}\|\eta_1\|_{0,\beta}.
\]
By Schwarz's inequality and \er{Bohm1}, we see that
the term ${\mathcal D}$ is bounded from below as
\[
{\mathcal D} 
\geq -(K^{-1}+1)u_+\|\psi\|_{0,\beta}^2-
2 (K^{1/2}+K^{-1/2})\|\psi\|_{0,\beta}\|\eta_1\|_{0,\beta}
-u_+\|\bm{\eta}\|_{0,\beta}^2
\geq d\|\Psi\|_{0,\beta}^2,
\]
where $d$ is a positive constant independent of $\beta$, $\phi_b$, and $t$.
Furthermore, by \er{Bohm1} and \er{ses1}, one can estimate the fifth term as
\begin{align*}
{}&
-\beta\int_{\Om} e^{\beta x_1}
\left\langle\!
A^1[\tV+\Psi]
\begin{bmatrix}
\nabla \cdot \hm{\eta}
\\
\nabla \psi
\end{bmatrix},
\begin{bmatrix}
\nabla \cdot \hm{\eta}
\\
\nabla \psi
\end{bmatrix}
\!\right\rangle\,dx
\\
&\geq -\beta\int_{\Om} e^{\beta x_1}
(u_+(\nabla \cdot \bm{\eta})^2+2\sqrt{K}(\nabla \cdot \bm{\eta}) \D_{x_1}\psi+u_+|\nabla \psi|^2)\,dx
-C(N_{m,\beta}(T)+|\phi_b|)\|\Psi\|_{1,\beta}^2
\\
%& \qquad \quad (\because \text{\er{ses1},\er{hypineq0}})
%\\
& \geq \beta d\|(\nabla \cdot \bm{\eta},\nabla \psi)\|_{0,\beta}^2
-C(N_{m,\beta}(T)+|\phi_b|)\|\Psi\|_{1,\beta}^2.
%\quad (\because \text{\er{Bohm1}})
\end{align*}
All terms on the left hand side except the first term 
has been estimated from below.

Substituting the above estimates into \er{bes10} leads to
\begin{align*}
{}&
\frac{d}{dt} \int_\Om e^{\beta x_1}
( |\Psi|^2+K^{-1}\psi^2
+ (\nabla \cdot \bm{\eta})^2
+ |\nabla \psi|^2) \,dx
+d\beta\|(\Psi,\nabla \psi,\nabla \cdot \bm{\eta})\|_{0,\beta}^2
\notag\\
&\leq  2\sqrt{D}\beta^2\|\Psi\|_{0,\beta}^2
+\mu\|\Psi\|_{0,\beta}^2
+C[\mu](N_{m,\beta}(T)+|\phi_b|)\|\Psi\|_{1,\beta}^2+C[\mu]|\phi_b|.
%\lb{bes11}
\end{align*}
To absorb the first term on the right hand side into the second term on the left hand side,
\footnote{We remark that here is only one place to choose $\beta$ suitably small and
hereafter we never change $\beta$ in the proofs of Theorems \ref{2.1} and \ref{2.2}.}{let}
us fix $\beta>0$ so small that %the following holds:
\begin{equation}\lb{defbeta1}
\beta \leq \min\{\al/2,d(4\sqrt{D})^{-1}\}.
\end{equation}
Then taking $\mu$, $N_{m,\beta}(T)$, and $|\phi_b|$ suitably small yields
\begin{align*}
{}&
\frac{d}{dt} \int_\Om e^{\beta x_1}
(|\Psi|^2+K^{-1}\psi^2
+(\nabla \cdot \bm{\eta})^2
+|\nabla \psi|^2) \,dx
+c\beta\|(\Psi,\nabla \psi,\nabla \cdot \bm{\eta})\|_{0,\beta}^2
\notag\\
&\leq C(N_{m,\beta}(T)+|\phi_b|)\|\nabla\Psi\|_{0,\beta}^2
+C|\phi_b|.
%\lb{bes12}
\end{align*}
Furthermore, multiplying this by $e^{\tilde{c} \beta t}$
and taking $\tilde{c}>0$ small enough, we have
\begin{align*}
{}&
e^{\tilde{c} \beta t}\|(\Psi,\nabla \psi,
\nabla \cdot \bm{\eta})(t)\|_{0,\beta}^2
+c\beta\int_0^t e^{\tilde{c} \beta \tau}
\|(\Psi,\nabla \psi,\nabla \cdot \bm{\eta})(\tau)\|_{0,\beta}^2\,d\tau
\notag \\
&\leq C\|\Psi_0\|_{1,\beta}^2
+\int_0^t e^{\tilde{c}\beta \tau}
\left(C(N_{m,\beta}(T)+|\phi_b|)\|\nabla\Psi(\tau)\|_{0,\beta}^2
+C|\phi_b|\right)\,d\tau
\notag \\
&\leq C\|\Psi_0\|_{1,\beta}^2
+\left(C(N_{m,\beta}(T)+|\phi_b|)
\sup_{t\in[0,T]}\|\nabla\Psi(t)\|_{0,\beta}^2
+C|\phi_b|\right)
\frac{1}{\tilde{c} \beta}(e^{\tilde{c} \beta t}-1),
%\lb{bes13}
\end{align*}
which immediately gives \er{basic1}.
\end{proof}

\subsection{Higher-order estimate}\lb{S4.3}
We estimate the higher order derivatives of $\Psi$ in this subsection.
Applying the operator $\D_x^{\bm{a}}$ with $|\bm{a}|=k$ for $k=1,\ldots,m$
to \er{re0}, we have
\begin{gather}
\D_t(\D_x^{\bm{a}}\Psi)+
\sum_{j=1}^3A^j[\tV+\Psi]\D_{x_j}(\D_x^{\bm{a}}\Psi)=
\begin{bmatrix}
0 \\ \D_x^{\bm{a}}\bm{h}
\end{bmatrix}
+I^{\bm{a}},
\lb{hes1}\\
I^{\bm{a}}:=\sum_{j=1}^3[\D_x^{\bm{a}},A^j[\tV+\Psi]]\D_{x_j}\Psi
%+\sum_{j=1}^3[\D_x^{\bm{a}},A^j[\Psi]]\D_{x_j}\Psi
+\begin{bmatrix}
0 \\ \D_x^{\bm{a}}\nabla \sigma
\end{bmatrix}
+\D_x^{\bm{a}}\left(B\Psi\right),
\notag
\end{gather}
where $[\D_x^{\bm{a}},\cdot]$ denotes a commutator.
Owing to \er{ses1}, $\beta\leq \alpha/2$, and $M\in H^\infty(\Om)$, it holds that
\begin{equation}\lb{h2}
\|\bm{h}\|_{m,\beta} \leq C|\phi_b|. 
\end{equation}
We also see from \er{eqiv0}, \er{ses1} and \er{ellineq4}
with the aid of  \er{A2} and \er{A4} in Appendix \ref{Appendix1} that
\begin{equation}\lb{I1}
\|I^{\bm{a}}\|_{0,\beta} 
\leq C(N_{m,\beta}(T)+|\phi_b|)\|\nabla^k\Psi\|_{0,\beta}
+C\|\Psi\|_{k-1,\beta}. 
\end{equation}

Let us now estimate the higher order derivatives of $\Psi$.
\begin{lem}\label{4.4}
Under the same assumption as in Proposition \ref{4.2},
it holds that 
\begin{equation}\lb{higher1}
\sup_{t\in[0,T]}\|\nabla^k\Psi(t)\|_{0,\beta}^2
\leq C(\|\Psi_0\|_{k,\beta}^2
+\sup_{t\in[0,T]}\|\Psi(t)\|_{k-1,\beta}^2+|\phi_b|)
\quad \text{$k = 1,\ldots,m$},
\end{equation}
where $C>0$ is a constant depending on $\beta$ but independent of $\phi_b$ and $t$.
\end{lem}
\begin{proof}
Take an inner product of \er{hes1} with $2e^{\beta x_1}\D_x^{\bm{a}}\Psi$,
and sum up the results for $\bm{a}$ with $|\bm{a}|=k$.
Then integrate the resultant equality by parts over $\Om$
and apply Gauss's divergence theorem to obtain
\begin{align}
{}&
\frac{d}{dt} \sum_{|\bm{a}|=k}\int_\Om e^{\beta x_1} |\D_x^{\bm{a}}\Psi|^2 \,dx
+\sum_{|\bm{a}|=k}\sum_{j=1}^3\int_{\D\Om} e^{\beta M(x_2,x_3)}
\langle n_j A^j[\tV+\Psi]\D_x^{\bm{a}}\Psi,\D_x^{\bm{a}}\Psi \rangle\,ds
\notag\\
&\quad
-\beta\sum_{|\bm{a}|=k}\int_{\Om} e^{\beta x_1}
\langle A^1[\tV+\Psi]\D_x^{\bm{a}}\Psi,\D_x^{\bm{a}}\Psi \rangle\,dx
\notag\\
&=\sum_{|\bm{a}|=k}\int_{\Om} e^{\beta x_1} \left(
\sum_{j=1}^3 \langle \{\D_{x_j}(A^j[\tV+\Psi])\}\D_x^{\bm{a}}\Psi,\D_x^{\bm{a}}\Psi \rangle
+\D_x^{\bm{a}}\bm{h}\cdot\D_x^{\bm{a}} \bm{\eta}
+\langle I^{\bm{a}}, \D_x^{\bm{a}}\Psi \rangle
\right) \,dx
\notag\\
&\leq C(N_{m,\beta}(T)+|\phi_b|+\mu)\|\nabla^k\Psi\|_{0,\beta}^2
+C[\mu]\|\Psi\|_{k-1,\beta}^2+C|\phi_b|,
\lb{hes4}
\end{align}
where we have used \er{ses1}, \er{h2}, \er{I1}, and Schwarz's inequality 
in deriving the last inequality.
Owing to \er{super1}, 
the second term on the left hand side is nonnegative and thus negligible.
The third term is bounded from below as
\begin{align}
{}&
-\beta\sum_{|\bm{a}|=k}\int_{\Om} e^{\beta x_1}
\langle A^1[\tV+\Psi]\D_x^{\bm{a}}\Psi,\D_x^{\bm{a}}\Psi \rangle \,dx
\notag\\
& \geq -\beta\sum_{|\bm{a}|=k}\int_{\Om} e^{\beta x_1}\{
(u_+(\D_x^{\bm{a}}\psi)^2
+2\sqrt{K}(\D_x^{\bm{a}}\psi)(\D_x^{\bm{a}}\eta_1)
+u_+|\D_x^{\bm{a}}\bm{\eta}|^2\}\,dx
\notag\\
&\qquad -C(N_{m,\beta}(T)+|\phi_b|)\|\nabla^k\Psi\|_{0,\beta}^2
\notag\\
& \geq c\beta \|\nabla^{k}\Psi\|_{0,\beta}^2
-C(N_{m,\beta}(T)+|\phi_b|)\|\nabla^k\Psi\|_{0,\beta}^2.
\lb{hes5}
\end{align}
The last inequality follows from \er{Bohm1}.

Substitute \er{hes5} into \er{hes4} and 
take $\mu$, $N_{m,\beta}(T)$, and $|\phi_b|$ suitably small to obtain
\begin{align*}
\frac{d}{dt}\sum_{|\bm{a}|=k} \int_\Om e^{\beta x_1} |\D_x^{\bm{a}}\Psi|^2 \,dx
+c\|\nabla^k \Psi\|_{0,\beta}^2
\leq C\|\Psi\|_{k-1,\beta}^2+C|\phi_b|.
\end{align*}
Multiplying this by $e^{\tilde{c} t}$ 
and letting $\tilde{c}>0$ be small enough, we have
\begin{align*}
{}&
e^{\tilde{c} t}\|\nabla^k\Psi(t)\|_{0,\beta}^2
+c\int_0^t e^{\tilde{c}\tau}  \|\nabla^k\Psi(\tau)\|_{0,\beta}^2 \,d\tau
\\
&\leq C\|\Psi_0\|_{k,\beta}^2
+C\int_0^t e^{\tilde{c}\tau}(
\|\Psi(\tau)\|_{k-1,\beta}^2+|\phi_b|) \,d\tau
\\
&\leq C\|\Psi_0\|_{k,\beta}^2
+C\left(\sup_{t \in [0,T]}\|\Psi(t)\|_{k-1,\beta}^2+|\phi_b|\right)
\frac{1}{\tilde{c}}\left(e^{\tilde{c} t}-1\right).
\end{align*}
This immediately completes \er{higher1}.
\end{proof}

\subsection{Completion of a priori estimate}\lb{S4.4}
We now complete the derivation of the a priori estimate \er{apes2}.
\begin{proof}[Proof of Proposition \ref{4.2}]
We begin by proving that
\begin{equation}\lb{apes0}
\sup_{t\in[0,T]} \|\Psi(t)\|_{m,\beta}^2 \leq C(\|\Psi_0\|_{m,\beta}^2+|\phi_b|).  
\end{equation}
Substituting \er{higher1} with $k=1$ into the right hand side of \er{basic1} and 
taking $N_{m,\beta}(T)$ and $|\phi_b|$ sufficiently small, we have
$\sup_{t\in[0,T]}\|\Psi(t)\|_{0,\beta}^2 \leq C(\|\Psi_0\|_{1,\beta}^2+|\phi_b|)$.
Then substituting this into the right hand side of \er{higher1} with $k=1$ leads to
$\sup_{t\in[0,T]}\|\Psi(t)\|_{1,\beta}^2 \leq C(\|\Psi_0\|_{1,\beta}^2+|\phi_b|)$.
Furthermore, the induction by using this and \er{higher1} yields \er{apes0}.

Note that the derivation of \er{apes2} is completed by showing that
\begin{gather}
\|\D_t\Psi(t)\|_{l,\beta}^2 \leq C\|\Psi(t)\|_{l+1,\beta}^2+C|\phi_b|^2,
\quad \text{for $l=0,\ldots,m-1$,}
\lb{hypineq3}
\end{gather}
because \er{ellineq4}, \er{apes0}, and \er{hypineq3} immediately give \er{apes2}.
Let us prove \er{hypineq3} for $l=0$. 
Multiply \er{re1} by $e^{\beta x_1/2}$, 
take the $L^2$-norm, and use \er{ellineq4} and \er{h2} to obtain
\begin{align*}
\|\partial_t \Psi\|_{0,\beta}
= \left\|\sum_{j=1}^3A^j[\tV+\Psi]\D_{x_j}\Psi
-\begin{bmatrix}
0 \\ \nabla \sigma
\end{bmatrix}
-B\Psi
-\begin{bmatrix}
0 \\ \bm{h}
\end{bmatrix} \right\|_{0,\beta}
\leq  C\|\Psi(t)\|_{1,\beta}+C|\phi_b|.
\end{align*}
Similarly, we deduce \er{hypineq3} for all $l\geq1$ by using \er{hes1}.
\end{proof}

\section{Construction of stationary solutions}\lb{S5}
This section is devoted to the construction of 
solutions $(\Psi^s,\sigma^s)$ to 
the stationary problem corresponding to \er{re0}.
It is to be expected from the bound \er{apes1} of time-global solutions $(\Psi,\sigma)$
that these global solutions may converge to 
some functions as $t$ tends to infinity.
Therefore, we define an sequence $\{(\Psi^k,\sigma^k)\}_{k=0}^\infty$ by
$(\Psi^k,\sigma^k)(t,x):=(\Psi,\sigma)(t+kT_*,x)$ for any $T^*>0$, and show that
this sequence converges to a time-periodic solution with a period $T^*$ 
to the problem of equations \er{re1} and \er{re2} 
with boundary conditions \er{re4} and \er{re5} in subsection \ref{S5.1}.
By using the arbitrary of period, 
it can be concluded in subsection \ref{S5.2} that 
the periodic solution is independent of time and 
thus the desired stationary solution.
The stability is also shown in subsection \ref{S5.3}.
It is reasonable to treat the time-periodic solution once,
because we need some convergence of the time derivative of $\Psi$
in passing to the limit in equations \er{re1}, 
but we may not be able to show directly that 
the time derivative converges to zero.

\subsection{Time-periodic solutions}\lb{S5.1}
\subsubsection{Uniqueness}\lb{S5.1.1}
We begin by studying the uniqueness of time-periodic solutions 
to problem \er{re1}--\er{re5} in the solution space 
\[
{\cal X}^m_{\beta}([0,T^*]):=
\left[L^\infty([0,T^*];H^{m}_{\beta}(\Om)) 
\cap  W^{1,\infty}([0,T^*];H^{m-1}_{\beta}(\Om))\right]^4
\times
C([0,T^*];H^{m+1}_{\beta}(\Om)).
\]
The uniqueness is summarized in the following proposition.
\begin{pro}\lb{5.1}
Let $u_+$ satisfy \er{Bohm1} and \er{asp1'}.
For $\beta>0$ being in Theorem \ref{4.1},
there exists $\delta_0>0$ such that 
if a time-periodic solution $(\Psi^*,\sigma^*) \in {\cal X}^3_{\beta}([0,T^*])$
with a period $T^*>0$ to problem \er{re1}--\er{re5}
exists and satisfies the following inequality, then it is unique.
\begin{equation}\lb{uniasp1}
\sup_{t\in[0,T^*]}(\|\Psi^*(t)\|_{3,\beta}
+\|\partial_t \Psi^*(t)\|_{2,\beta}+\|\sigma^*(t)\|_{4,\beta})
+|\phi_b| \leq \delta_0.
\end{equation}
\end{pro}

Let $(\Psi,\sigma)$ and $(\Psi^*,\sigma^*)$ 
be time-periodic solutions with \er{uniasp1}, where
$\Psi={}^{t}(\psi,\bm{\eta})$ and $\Psi^*={}^{t}(\psi^*,\bm{\eta}^*)$.
It is straightforward to see 
that $\bar{\Psi}={}^{t}(\bar{\psi},\bar{\bm{\eta}}):={}^{t}(\psi-\psi^*,\bm{\eta}-\bm{\eta}^*)$ 
and $\bar{\sigma}:=\sigma-\sigma^*$ satisfy
\begin{subequations}\lb{unieq0}
\begin{gather}
\D_t\bar{\Psi}
+\sum_{j=1}^3A^j[\tV+\Psi]\D_{x_j}\bar{\Psi}
=\begin{bmatrix}
0 \\ \nabla \bar{\sigma}
\end{bmatrix}
+B[\tV',\nabla M]\bar{\Psi}
-\sum_{j=1}^3\bar{\eta}_j\D_{x_j}\Psi^*,
\lb{unieq1}
\\
\Delta \bar{\sigma}-\bar{\sigma}=K^{-1/2}\bar{\psi}
+g_0[\psi,\tv]-g_0[\psi^*,\tv]
+g_1[\sigma,\tphi]-g_1[\sigma^*,\tphi],
\lb{unieq2}
\\
\lim_{|x|\to\infty}(\bar{\Psi},\bar{\sigma})(x)=0, 
\lb{unieq3}
\\
\bar{\sigma}(M(x_2,x_3),x_2,x_3)=0.
\lb{unieq4}
\end{gather}
\end{subequations}
Note that the essential difference between systems \er{re1} and \er{unieq1}
is only the rightmost of these equations.
For equations \er{re2} and \er{unieq2}, 
the terms $g_0[\psi,\tv]$, $g_1[\phi,\tphi]$, and $g_2[\tphi',\nabla M]$ 
are just replaced by $g_0[\psi,\tv]-g_0[\psi^*,\tv]$, 
$g_1[\sigma,\tphi]-g_1[\sigma^*,\tphi]$, and zero, respectively.
Therefore, the calculations in Section \ref{S4} also work for \er{unieq0}
by adjusting them slightly.

For the proof of Proposition \ref{5.1}, we first derive 
estimates of $\bar{\sigma}$.
\begin{lem}\lb{uniell1}
Under the same assumption as in Proposition \ref{5.1}, it holds that
\begin{gather}
\|\bar{\sigma}(t)\|_{1,\beta}^2 
\leq (K^{-1}+D\beta^2+C\delta_0) \|\bar{\psi}(t)\|_{0,\beta}^2,
\lb{uniellineq1}\\
\|\bar{\sigma}(t)\|_{2,\beta}^2 \leq C \|\bar{\psi}(t)\|_{0,\beta}^2,
\lb{uniellineq4}
\end{gather}
where $D$ is the same positive constant being in Lemma \ref{ell1} and
$C$ is a positive constant independent of $\beta$, $\phi_b$, and $t$.
\end{lem}
\begin{proof}
This follows by the same method as in the proof of Lemma \ref{ell1}.
Indeed, we only need to replace $(N_{m,\beta}(T)+|\phi_b|)$ and $g_2$ 
by $\delta_0$ and zero, respectively, and use the inequalities
\begin{gather*}
\|g_0[\psi,\tv]-g_0[\psi^*,\tv]\|_{0,\beta}
\leq C \delta_0 \|\bar{\psi}\|_{0,\beta}, 
\quad
\|g_1[\sigma,\tphi]-g_1[\sigma^*,\tphi]\|_{0,\beta} 
\leq C \delta_0 \|\bar{\sigma}\|_{0,\beta}
\end{gather*}
instead of \er{g0} and \er{g1}.
\end{proof}

We are now at a position to show Proposition \ref{5.1}.

\begin{proof}[Proof of Proposition \ref{5.1}]
It suffices to prove $\bar{\Psi}=0$,
since this and \er{uniellineq4} lead to $\bar{\sigma}=0$.
To this end, we only need to show
\begin{gather}
\int_0^{T^*} \|\bar{\Psi}(\tau)\|_{0,\beta}^2\,d\tau
\leq C\delta_0 \int_0^{T^*} \|\nabla\bar{\Psi}(\tau)\|_{0,\beta}^2\,d\tau,
\lb{unibasic1}\\
\int_0^{T^*} \|\nabla \bar{\Psi}(\tau)\|_{0,\beta}^2\,d\tau
\leq C\int_0^{T^*} \|\bar{\Psi}(\tau)\|_{0,\beta}^2\,d\tau.
\lb{unihigher1}
\end{gather}
In fact, one can deduce $\bar{\Psi}=0$ 
by substituting \er{unihigher1} into the right hand side of \er{unibasic1}
and taking $\de_0$ sufficiently small.

Let us first derive \er{unibasic1}.
In much the same way as the derivation of the equality in \er{bes10} from \er{re0}, 
we see from \er{unieq0} that
\begin{align}
{}&
\frac{d}{dt} \int_\Om e^{\beta x_1}
\left(|\bar{\Psi}|^2+K^{-1}\bar{\psi}^2
+(\nabla \cdot \bar{\bm{\eta}})^2
+|\nabla \bar{\psi}|^2\right)\,dx
\notag\\
&
\quad+\sum_{j=1}^3\int_{\D\Om} e^{\beta M(x_2,x_3)}\left(
\langle n_j A^j[\tV+\Psi]\bar{\Psi},\bar{\Psi} \rangle 
+\left\langle\!
n_jA^j[\tV+\Psi]
\begin{bmatrix}
\nabla \cdot \bar{\hm{\eta}}
\\
\nabla \bar{\psi}
\end{bmatrix},
\begin{bmatrix}
\nabla \cdot \bar{\hm{\eta}}
\\
\nabla \bar{\psi}
\end{bmatrix}
\!\right\rangle\right)\,ds
\notag\\
&
\quad+\sum_{j=1}^3\int_{\D\Om} e^{\beta M(x_2,x_3)}n_j K^{-1}(\eta_j+\tu_j)\bar{\psi}^2\,ds
\notag\\
&
\quad-\beta\int_{\Om} e^{\beta x_1}
(\langle A^1[\tV+\Psi]\bar{\Psi},\bar{\Psi} \rangle
+K^{-1}(\eta_1+\tu_1)\bar{\psi}^2
-2\bar{\sigma}\bar{\eta}_1)\,dx
\notag\\
&
\quad-\beta\int_{\Om} e^{\beta x_1}
\left\langle\!
A^1[\tV+\Psi]
\begin{bmatrix}
\nabla \cdot \bar{\hm{\eta}}
\\
\nabla \bar{\psi}
\end{bmatrix},
\begin{bmatrix}
\nabla \cdot \bar{\hm{\eta}}
\\
\nabla \bar{\psi}
\end{bmatrix}
\!\right\rangle\,dx
\notag\\
&=\int_{\Om} e^{\beta x_1}\bar{\mathcal R}\,dx,
%\notag\\
%&\leq C(N_{m,\beta}(T)+|\phi_b|)\|\Psi\|_{1,\beta}^2+C|\phi_b|,
\lb{unibes10}
\end{align}
where $\bar{\mathcal R}$ is defined as
\begin{align*}
\bar{\mathcal R}:=&
\sum_{j=1}^3 \langle \{\D_{x_j}(A^j[\tV+\Psi])\}\bar{\Psi},\bar{\Psi} \rangle
+2\langle B\bar{\Psi}, \bar{\Psi}\rangle
-2\sum_{j=1}^3\langle \bar{\eta}_j\D_{x_j}\Psi^*, \bar{\Psi}\rangle
\\
& +\sum_{j=1}^3 \!\left\langle\!
\{\D_{x_j}(A^j[\tV+\Psi])\}
\begin{bmatrix}
\nabla \cdot \bar{\hm{\eta}}
\\
\nabla \bar{\psi}
\end{bmatrix},
\begin{bmatrix}
\nabla \cdot \bar{\hm{\eta}}
\\
\nabla \bar{\psi}
\end{bmatrix}
\!\right\rangle\!
+2\{\nabla(B\bar{\Psi})_1
-\nabla(\tilde{\bm{u}}+\bm{\eta})\nabla \bar{\psi}
-\nabla\bar{\bm{\eta}}\nabla \psi^*\}
\cdot\nabla\bar{\psi}
\\
&+2\left\{\D_{x_1}(B\bar{\Psi})_{2}
-\sum_{i,j=1}^3\{\D_{x_i}(\tu_j+\eta_j)\} \D_{x_j}\bar{\eta}_{i}
-\sum_{i,j=1}^3(\D_{x_i}\bar{\eta}_j) \D_{x_j}\eta^*_{i}\right\}
(\nabla \cdot \bar{\bm{\eta}})
\\
&+2\sum_{j=1}^3\!\left\langle\!
\bar{\eta}_j
\begin{bmatrix}
\D_{x_j}\nabla \cdot \hm{\eta}^* \\ \D_{x_j}\nabla\psi^*
\end{bmatrix},
\begin{bmatrix}
\nabla \cdot \bar{\hm{\eta}}
\\
\nabla \bar{\psi}
\end{bmatrix}
\!\right\rangle\!
+K^{-1}\sum_{j=1}^3\left\{\D_{x_j}(\eta_j+\tu_j) \right\} \bar{\psi}^2
+2K^{-1}\bar{\psi}(B\bar{\Psi})_1
\\
&-2K^{-1}\bar{\psi}\bar{\bm{\eta}}\cdot\nabla\psi^*
+2(g_0[\psi,\tv]-g_0[\psi^*,\tv]
+g_1[\sigma,\tphi]-g_1[\sigma^*,\tphi])(\nabla\cdot\bar{\bm{\eta}}).
\end{align*}

Note that we must be careful to handle the terms 
having $\psi^*$ and $\bm{\eta}^*$ in $\bar{\mathcal R}$,
since some of these include the second-order derivatives.
Using \er{ses1} and \er{uniasp1}, we estimate $\bar{\mathcal R}$ as
\begin{equation}
|\bar{\mathcal R}|\leq
C\delta_0|(\bar{\Psi},\bar{\sigma},\nabla\bar{\Psi})|^2
+C|\bar{\Psi}||\nabla\bar{\Psi}||\nabla^2{\Psi^*}|.  
\end{equation}
Then Sobolev's and H\"older's inequalities give
\begin{align}\lb{unibes9}
\left|\int_{\Om} e^{\beta x_1}\bar{\mathcal R}\,dx\right|
\leq C\delta_0(\|\bar{\Psi}\|_{1,\beta}^2+\|\bar{\sigma}\|_{1,\beta}^2)
\leq C\delta_0\|\bar{\Psi}\|_{1,\beta}^2,
\end{align}
where we have also used \er{uniellineq4} in deriving the last inequality.

On the other hand, we notice that the left hand side of \er{unibes10} has 
the same form as that of \er{bes10}.
Therefore, the second and third terms are nonnegative and so negligible 
if $\de_0$ is sufficiently small.
Furthermore, with the aid of \er{uniellineq1}, 
the fourth and fifth terms are bounded from below as
\begin{align*}
{}&
-\beta\int_{\Om} e^{\beta x_1}
(\langle A^1[\tV+\Psi]\bar{\Psi},\bar{\Psi} \rangle
+K^{-1}(\eta_1+\tu_1)\bar{\psi}^2
-2\bar{\sigma}\bar{\eta}_1)\,dx
\\
&\geq \beta d\|\bar{\Psi}\|_{0,\beta}^2
-(2\sqrt{D}\beta^2+\mu+C[\mu]\de_0)\|\bar{\Psi}\|_{0,\beta}^2
\end{align*}
and
\begin{gather*}
-\beta\int_{\Om} e^{\beta x_1}
\left\langle\!
A^1[\tV+\Psi]
\begin{bmatrix}
\nabla \cdot \bar{\hm{\eta}}
\\
\nabla \bar{\psi}
\end{bmatrix},
\begin{bmatrix}
\nabla \cdot \bar{\hm{\eta}}
\\
\nabla \bar{\psi}
\end{bmatrix}
\!\right\rangle\,dx 
\geq \beta d\|(\nabla \bar{\psi},\nabla \cdot \bar{\bm{\eta}})\|_{0,\beta}^2
-C\de_0\|\Psi\|_{1,\beta}^2,
\end{gather*}
where $d$ and $D$ are the same positive constants as in \er{defbeta1}.
Substitute these inequalities and \er{unibes9} into \er{unibes10}, use \er{defbeta1}, 
and let $\mu$ and $\de_0$ be small enough to obtain
\begin{align}
\frac{d}{dt} \int_\Om e^{\beta x_1}
(|\bar{\Psi}|^2+K^{-1}\bar{\psi}^2
+(\nabla \cdot \bar{\bm{\eta}})^2
+|\nabla \bar{\psi}|^2) \,dx
+c\beta\|(\bar{\Psi},\nabla \bar{\psi},\nabla \cdot \bar{\bm{\eta}})\|_{0,\beta}^2
\leq C\de_0\|\nabla\bar{\Psi}\|_{0,\beta}^2.
\lb{unibes12}
\end{align}
Then integrating this over $[0,T^*]$ and using the periodicity of solutions, 
we conclude \er{unibasic1}.

Let us complete the proof by showing \er{unihigher1}.
Apply $\D_x^{\bm{a}}$ with $|\bm{a}|=1$ to \er{unieq1},
take an inner product of this with $2e^{\beta x_1}\D_x^{\bm{a}}\bar{\Psi}$,
and sum up the results for $\bm{a}$ with $|\bm{a}|=1$.
Then integrating the resultant equality over $\Om$ 
and applying Gauss's divergence theorem, we have
\begin{align}
{}&
\frac{d}{dt} \sum_{|\bm{a}|=1}\int_\Om e^{\beta x_1}
|\D_x^{\bm{a}}\bar{\Psi}|^2 \,dx
+\sum_{|\bm{a}|=1}\sum_{j=1}^3\int_{\D\Om} e^{\beta M(x_2,x_3)}
\langle n_j A^j[\tV+\Psi]\D_x^{\bm{a}}\bar{\Psi},\D_x^{\bm{a}}\bar{\Psi} \rangle\,ds
\notag\\
&\quad
-\beta\sum_{|\bm{a}|=1}\int_{\Om} e^{\beta x_1}
\langle A^1[\tV+\Psi]\D_x^{\bm{a}}\bar{\Psi},\D_x^{\bm{a}}\bar{\Psi} \rangle\,dx
\notag\\
&=\sum_{|\bm{a}|=1}\int_{\Om} e^{\beta x_1} \left(
\sum_{j=1}^3 \langle \{\D_{x_j}(A^j[\tV+\Psi])\}\D_x^{\bm{a}}\bar{\Psi},\D_x^{\bm{a}}\bar{\Psi} \rangle
+2(\nabla\D_x^{\bm{a}}\bar{\sigma})\cdot\D_x^{\bm{a}}\bar{\bm{\eta}}
+2\langle \D_x^{\bm{a}}(B\bar{\Psi}), \D_x^{\bm{a}}\bar{\Psi} \rangle
\right.
\notag\\
&\mspace{150mu} \left.-2\sum_{j=1}^3\langle\{\D_x^{\bm{a}}(A^j[\tV+\Psi])\}\D_{x_j}\bar{\Psi},\D_x^{\bm{a}}\bar{\Psi} \rangle
-2\sum_{j=1}^3\langle\D_{x}^{\bm{a}}(\bar{\eta}_j\D_{x_j}\Psi^*),\D_x^{\bm{a}}\bar{\Psi} \rangle\right)\,dx
\notag\\
&\leq (C\delta_0+\mu)\|\bar{\Psi}\|_{1,\beta}^2
+C[\mu]\|\nabla^2\bar{\sigma}\|_{0,\beta}^2
\notag\\
&\leq (C\delta_0+\mu)\|\nabla\bar{\Psi}\|_{0,\beta}^2
+C[\mu]\|\bar{\Psi}\|_{0,\beta}^2, 
\lb{unihes4}
\end{align}
where we have estimated the right hand side of 
the above equality similarly to \er{unibes9}.
The left hand side of the equality in \er{unihes4}
has the same form as that of \er{hes4}.
Therefore, the second term on the left hand side is nonnegative.
The third term is bounded from below as
\begin{align*}
-\beta\sum_{|\bm{a}|=1}\int_{\Om} e^{\beta x_1}
\langle A^1[\tV+\Psi]\D_x^{\bm{a}}\bar{\Psi},\D_x^{\bm{a}}\bar{\Psi} \rangle \,dx
\geq c\beta \|\nabla\bar{\Psi}\|_{0,\beta}^2
-C\de_0\|\nabla\bar{\Psi}\|_{0,\beta}^2.
\end{align*}
Substitute this into \er{unihes4} and
let $\mu$ and $\de_0$ be sufficiently small to get
\begin{equation}\lb{unihes5}
\frac{d}{dt} \sum_{|\bm{a}|=1}\int_\Om e^{\beta x_1}|\D_x^{\bm{a}}\bar{\Psi}|^2 \,dx
+c\|\nabla\bar{\Psi}\|_{0,\beta}^2
\leq C\|\bar{\Psi}\|_{0,\beta}^2. 
\end{equation}
Then integrating this over $[0,T_*]$ and 
using the periodicity of solutions, we conclude \er{unihigher1}.
%The proof is complete.
\end{proof}

\subsubsection{Existence}\lb{S5.1.2}
For the construction of time-periodic solutions, we define
\[
(\Psi^k,\sigma^k)(t,x):=(\Psi,\sigma)(t+kT^*,x)
\quad \text{for $k=1,2,3,\ldots$,} 
\]
where $(\Psi,\sigma)$ is the time-global solution in Theorem \ref{4.1}
and $\Psi^k$ denotes ${}^t(\psi^k,\bm{\eta}^k)$.
Let us start from discussing the next lemma.

\begin{lem}\lb{5.2}
Let $u_+$ satisfy \er{Bohm1} and \er{asp1'}.
For $\beta>0$ being in Theorem \ref{4.1} and any $T^*>0$,
there exists $\gamma>0$ and $C>0$ independent of $k$ and $T^*$ such that
\begin{equation}\lb{exiapes0}
\|(\Psi-\Psi^k)(t)\|_{1,\beta}
+\|(\sigma-\sigma^k)(t)\|_{2,\beta} \leq Ce^{-\gamma t}
\quad \text{for $k=1,2,3,\ldots$.}
\end{equation}
\end{lem}
\begin{proof}
We note that the time-global solution in Theorem \ref{4.1} satisfies \er{apes1}.
Therefore, by the same method as in the derivations of \er{unibes12} and \er{unihes5},
one can see that
\begin{align*}
{}&
\frac{d}{dt} \int_\Om e^{\beta x_1}
(|\Psi-\Psi^k|^2+K^{-1}|\psi-\psi^k|^2
+|\nabla \cdot(\bm{\eta}-\bm{\eta}^k)|^2
+|\nabla (\psi-\psi^k)|^2) \,dx
\\
&+c\beta\|(\Psi-\Psi^k,\nabla \cdot (\bm{\eta}-\bm{\eta}^k),\nabla(\psi-\psi^k))\|_{0,\beta}^2
\leq C(\|\Psi_0\|_{m,\beta}+|\phi_b|^{1/2})\|\nabla(\Psi-\Psi^k)\|_{0,\beta}^2
\end{align*}
and
\begin{gather*}
\frac{d}{dt} \sum_{|\bm{a}|=1}\int_\Om e^{\beta x_1}|\D_x^{\bm{a}}(\Psi-\Psi^k)|^2 \,dx
+c\|\nabla(\Psi-\Psi^k)\|_{0,\beta}^2
\leq C\|\Psi-\Psi^k\|_{0,\beta}^2. 
\lb{exihes5}
\end{gather*}
Then multiply these two by $e^{\tilde{c}t}$, 
integrate the results over $[0,T^*]$, 
and take $\tilde{c}>0$ suitably small to get
\begin{align*}
{}& e^{\tilde{c}t}\|(\Psi-\Psi^k)(t)\|_{0,\beta}^2
+\int_0^t e^{\tilde{c}\tau}\|(\Psi-\Psi^k)(\tau)\|_{0,\beta}^2\,d\tau
\\
&\leq C\|(\Psi-\Psi^k)(0)\|_{1,\beta}^2
+C(\|\Psi_0\|_{m,\beta}+|\phi_b|^{1/2})
\int_0^t e^{\tilde{c}\tau}\|\nabla(\Psi-\Psi^k)(\tau)\|_{0,\beta}^2\,d\tau
\end{align*}
and
\begin{align*}
{}& e^{\tilde{c}t}\|\nabla(\Psi-\Psi^k)(t)\|_{0,\beta}^2
+\int_0^t e^{\tilde{c}\tau}\|\nabla(\Psi-\Psi^k)(\tau)\|_{0,\beta}^2\,d\tau
\\
&\leq C\|(\Psi-\Psi^k)(0)\|_{1,\beta}^2
+C\int_0^t e^{\tilde{c}\tau}\|(\Psi-\Psi^k)(\tau)\|_{0,\beta}^2\,d\tau.
\end{align*}
From these two and \er{apes1},
we have the estimate of $\Psi-\Psi^k$ in \er{exiapes0}
by taking $\|\Psi_0\|_{m,\beta}$ and $|\phi_b|$ suitably small again if necessary.
Now it remains to obtain the estimate of $\sigma-\sigma^k$ in \er{exiapes0}.
The same proof as Lemma \ref{uniell1} works for $\sigma-\sigma^k$ and thus 
$\|\sigma-\sigma^k\|_{2,\beta}\leq C\|\psi-\psi^k\|_{0,\beta}$ holds.
This immediately completes the proof.
\end{proof}

We are now in a position to construct time-periodic solutions to 
problem \er{re1}--\er{re5}.
\begin{pro}\lb{5.3}
Let $m \geq 3$ and $u_+$ satisfy \er{Bohm1} and \er{asp1'}.
For $\beta>0$ being in Theorem \ref{4.1} and any $T^*>0$,
there exists a constant $\delta>0$ independent of $T^*$ such that 
if $|\phi_b| \leq \delta$, then the problem \er{re1}--\er{re5} 
has a time-periodic solution $(\Psi^*,\sigma^*)\in {\mathcal X}^m_\beta([0,T^*])$
with a period $T^*>0$. Furthermore, it satisfies
\begin{equation}\lb{apes3}
\sup_{t\in[0,T^*]}(\|\Psi^*(t)\|_{m,\beta}
+\|\partial_t \Psi^*(t)\|_{m-1,\beta}+\|\sigma^*(t)\|_{m+1,\beta})
\leq C|\phi_b|^{1/2},
\end{equation}
where $C>0$ is a constant independent of $T^*$.
\end{pro}
\begin{proof}
First of all, applying Theorem \ref{4.1} and Lemma \ref{5.2} to 
initial--boundary value problem \er{re0},
we have the time-global solution $(\Psi,\sigma)$ satisfying
\er{apes1} and \er{exiapes0}.
Let us start from proving that $\{(\Psi^k,\sigma^k)\}$ is a Cauchy sequence in the Banach space 
$[\cap_{i=0}^1C^i([0,T^*];H^{m-i-1}_\beta(\Om))]^4 \times C([0,T^*];H^{m+1}_\beta(\Om))$.
For the case $k>k'$, it follows from \er{exiapes0} that
\begin{align*}
\sup_{t \in [0,T^*]} \|(\Psi^k-\Psi^{k'},\sigma^k-\sigma^{k'})(t) \|_{0,\beta}
&=\sup_{t \in [k'T^*,(k'+1)T^*]} \|(\Psi-\Psi^{k-k'},\sigma-\sigma^{k-k'})(t)\|_{0,\beta}
\\
&\leq C e^{-\ga k'T^*}.
\end{align*}
We see from this and \er{apes1} 
with the aid of \er{eqiv0} and the Gagliardo--Nirenberg inequalities that
\begin{align}
\sup_{t \in [0,T^*]} \|(\Psi^k-\Psi^{k'})(t) \|_{m-1,\beta}
\leq  C e^{-\ga k'T^*/m},
\lb{exies3}\\
\sup_{t \in [0,T^*]} \|(\sigma^k-\sigma^{k'})(t) \|_{m+1,\beta}
\leq C e^{-\ga k'T^*/(m+2)}. 
\lb{exies4}
\end{align}
It remains to show that $\{\Psi^k\}$ is a Cauchy sequence 
in $C^1([0,T^*];H^{m-2}(\Om))$. 
It is straightforward to obtain from \er{re1} and \er{apes1} that
\begin{align*}
{}&
|\D_t (\Psi^k-\Psi^{k'})|
\\
&= \left|
\sum_{j=1}^3A^j[\tV+\Psi^k]\D_{x_j}(\Psi^k-\Psi^{k'})
-\begin{bmatrix}
0 \\ \nabla (\sigma^k-\sigma^{k'})
\end{bmatrix}
-B(\Psi^k-\Psi^{k'})
+\sum_{j=1}^3(\eta^k_j-\eta^{k'}_j)\D_{x_j}\Psi^{k'}
\right|
\\
& \leq C |((\Psi^k-\Psi^{k'}),\nabla(\Psi^k-\Psi^{k'}),\nabla(\sigma^k-\sigma^{k'}))|,
\end{align*}
which gives
\begin{equation*}
\|\D_t (\Psi^k-\Psi^{k'})\|_{0,\beta} \leq  C e^{-\ga k'T^*/(m+2)}.
\end{equation*}
This and \er{apes1} together with the Gagliardo-Nirenberg inequalities yield
\begin{align}
\sup_{t \in [0,T^*]} \|\D_t(\Psi^k-\Psi^{k'})(t) \|_{m-2,\beta}
\leq  C e^{-\ga k'T^*/\{(m+2)(m-1)\}}.
\notag
\end{align}
Therefore, $\{(\Psi^k,\sigma^k)\}$ is a Cauchy sequence 
%in $[\cap_{i=0}^1C^i([0,T^*];H^{m-i-1}_\beta(\Om))]^4 \times C([0,T^*];H^{m+1}_\beta(\Om))$
and then there exists a limit  $(\Psi^*,\sigma^*)$ such that
\begin{equation}\lb{converge1}
(\Psi^k,\sigma^k) \to (\Psi^*,\sigma^*) \quad \text{in} \ \
\left[\bigcap_{i=0}^1C^i([0,T^*];H^{m-i-1}_\beta(\Om))\right]^4
\times C([0,T^*];H^{m+1}_\beta(\Om)).
\end{equation}

The limit $(\Psi^*,\sigma^*)$ obviously satisfies \er{re1}--\er{re5}.
Let us check that $(\Psi^*,\sigma^*)$ is a time-periodic function with the period $T^*$.
The sequences $(\Psi^k,\sigma^k)(T^*,x)$ and $(\Psi^{k+1},\sigma^{k+1})(0,x)$ 
converges to $(\Psi^*,\sigma^*)(T^*,x)$ and $(\Psi^*,\sigma^*)(0,x)$, respectively, 
as $k$ tends to infinity.
We notice that $(\Psi^k,\sigma^k)(T^*,x)=(\Psi^{k+1},\sigma^{k+1})(0,x)$ holds 
and so does $(\Psi^*,\sigma^*)(T^*,x)=(\Psi^*,\sigma^*)(0,x)$.
Consequently, $(\Psi^*,\sigma^*)$ is a time-periodic solution to problem \er{re1}--\er{re5}.

We complete the proof by showing that $(\Psi^*,\sigma^*)$ belongs to ${\cal X}^m_\beta(0,T)$
and satisfies \er{apes3}. 
The function $\sigma^*$ already has enough regularity and then \er{apes1} implies that
\begin{equation}\lb{exies7}
\sup_{t\in [0,T^*]}\|\sigma^*(t)\|_{m+1,\beta} 
\leq C(\|\Psi_0\|_{m,\beta}+|\phi_b|^{1/2}).
\end{equation}
On the other hand, by a standard method for hyperbolic systems 
(for instance, see \cite[Section 5]{Rac}), we see from \er{apes1} that
$\Psi^k(t)$ converges to $\Psi^*(t)$ weakly in $H^m_\beta(\Om)$ 
for each $t \in [0,T^*]$. It also holds that
\begin{equation}\lb{exies8}
\sup_{t\in [0,T^*]}\|\Psi^*(t)\|_{m,\beta} 
\leq C(\|\Psi_0\|_{m,\beta}+|\phi_b|^{1/2}),
\end{equation}
which means $\Psi^* \in L^\infty(0,T^*;H^{m}_\beta(\Om))$.
Similarly, it holds that
%From these facts and system \er{re1}, it is easy to check that
$\D_t \Psi^* \in L^\infty(0,T^*;H^{m-1}_\beta(\Om))$ and 
\begin{equation}\lb{exies9}
\sup_{t\in [0,T^*]}\|\D_t \Psi^*(t)\|_{m-1,\beta} 
\leq C(\|\Psi_0\|_{m,\beta}+|\phi_b|^{1/2}).
\end{equation}
Hence, the time-periodic solution $(\Psi^*,\sigma^*)$ belongs ${\cal X}^m_\beta(0,T)$
in which the uniqueness has been shown.
It remains to obtain \er{apes3}. For the initial data $\Psi_0=0$,
we have another time-periodic solution  by the above method.
Proposition \ref{5.1} together with estimates \er{exies7}--\er{exies9}
ensures that both periodic solutions are same.
Therefore, by plugging $\Psi_0=0$ into \er{exies7}--\er{exies9}, we have \er{apes3}.
\end{proof}

\subsection{Stationary solutions}\lb{S5.2}

We show that the time-periodic solutions constructed in Subsection \ref{S5.1}
are time-independent.

\begin{proof}[Proof of Theorem \ref{2.1}]
Proposition \ref{5.3} ensures the existence of
time-periodic solutions $(\Psi^*,\sigma^*)$
of problem \er{re1}--\er{re5} for any period $T^*$.
We remark that the smallness assumption for 
the boundary data $\phi_b$ is independent of the period $T^*$.
Hence, one can have time-periodic solutions $(\Psi^*,\sigma^*)$ 
with the period $T^*$ and
$(\Psi^*_l,\sigma^*_l)$ with the period $T^*/2^l$ for $l \in \mathbb N$ 
under the same assumption for $\phi_b$.
Furthermore, $(\Psi^*,\sigma^*)=(\Psi^*_l,\sigma^*_l)$ follows from 
Proposition \ref{5.1}, since both $(\Psi^*,\sigma^*)$ and $(\Psi^*_l,\sigma^*_l)$ are 
the time-periodic solutions with the period $T^*$ and satisfy \er{apes3}. 
Hence, we see that
\[
(\Psi^*,\sigma^*) \left(0,x\right)=
(\Psi^*,\sigma^*) \left(\frac{i}{2^l}T^*,x\right)
\quad \text{for $i=1,2,3,\ldots,2^l$ and $l=0,1,2,\ldots$.} 
\]
Because the set 
$\cup_{l \geq 0} \{{i}/{2^l} \ ; \ i=1,2,3,\ldots,2^l\}$
is dense in $[0,T^*]$,
we see from the continuity of $(\Psi^*,\sigma^*)$
that $(\Psi^*,\sigma^*)$ is independent of $t$.
Therefore, $(\Psi^s,\sigma^s)=(\Psi^*,\sigma^*)$
is the desired stationary solution. 
\end{proof}

\subsection{Stability in the exponential weighted Sobolev space}\lb{S5.3}
This subsection is devoted to the completion of the proof of Theorem \ref{2.2}.
Since the time-global solutions to problem \er{re0} has been constructed in Theorem \ref{4.1},
it suffices to show the asymptotic stability of stationary solutions.

\begin{proof}[Proof of Theorem \ref{2.2}]
Theorem \ref{4.1} and Lemma \ref{5.2} ensure that
initial--boundary value problem \er{re0} has 
a unique time-global solution satisfying \er{apes1} and \er{exiapes0}
if $\|\Psi_0\|_{m,\beta}$ and $|\phi_b|$ are small enough.
%So, it suffices to show that this time-global solution $(\Psi,\sigma)$ converges 
%to the stationary solution solution $(\Psi^s,\sigma^s)$ exponentially fast 
%as $t$ tends to infinity.
Passing to the limit $k\to \infty$ in \er{exiapes0},
we have $\|(\Psi-\Psi^s,\sigma-\sigma^s)(t)\|_{0,\beta} \leq C e^{-\ga t}$
thanks to \er{converge1} and $(\Psi^s,\sigma^s)=(\Psi^*,\sigma^*)$.
Then this inequality and \er{apes1} together with the Gagliardo-Nirenberg inequalities 
give the decay estimate \er{decay1}.
%\begin{equation*}%\lb{expdecay1}
%\sup_{x\in \Om}|(\Psi-\Psi^s,\sigma-\sigma^s)(t,x)| \leq C e^{-\ga t}.
%\end{equation*}
The proof is complete.
\end{proof}

\medskip

\noindent
{\bf Acknowledge.} 
This work was supported by JSPS KAKENHI Grant Numbers 26800067 and 18K03364.

\begin{appendix}
\section{Stability in the algebraic weighted Sobolev space}\lb{S6}
In this section, we give an outline of the proof of Theorem \ref{1.4}
which states the stability of stationary solutions $(\rho^s,\bm{u}^s,\phi^s)$ 
in the algebraic weighted Sobolev space.
By following \cite{NOS}, we introduce new functions 
\[
w^s:=\log\rho^s, \quad w:=\log\rho 
\]
and perturbations $(\overline{\psi},\overline{\bm{\eta}},\overline{\sigma})$ 
from the stationary solution:
\[
(\overline{\psi},\overline{\bm{\eta}},\overline{\sigma})
:=(w-w^s,\bm{u}-\bm{u}^s,\phi-\phi^s).
\]
Then it is seen that $(\overline{\psi},\overline{\bm{\eta}},\overline{\sigma})$ satisfies
the system of equations
%\begin{subequations}\label{stare0}
 \begin{align*}
  \D_t\overline{\psi} + (\bm{u}^s+\overline{\bm{\eta}}) \cdot \nabla \overline{\psi} 
  + {\rm{div}} \overline{\bm{\eta}} + \overline{\bm{\eta}}\cdot\nabla w^s =& 0,
  %\label{stare1}
  \\
  \D_t\overline{\bm{\eta}} + (\bm{u}^s+\overline{\bm{\eta}}) \cdot \nabla \overline{\bm{\eta}} 
  +K \nabla \overline{\psi} - \nabla \overline{\sigma} + \overline{\bm{\eta}} \cdot \nabla \bm{u}^s =& 0,
  %\label{stare2}
  \\
  \Delta \overline{\sigma}
  -e^{(w^s+\overline{\psi})}+e^{w^s}
  +e^{-\phi^s-\overline{\sigma}}-e^{-\phi^s}=& 0, \quad
  (t,x)\in {\mathbb R}_+\times \Om
  % \label{stare3}
 \end{align*}
%\end{subequations}
with the initial and boundary data
\begin{gather*}
  (\overline{\psi},\overline{\bm{\eta}})(0,x)=(\log\ro_0-w^s,\bm{u}_0-\bm{u}^s)(x),
  \\ %\label{staini1}\\
  \lim_{|x|\rightarrow \infty} (\overline{\psi},\overline{\bm{\eta}},\overline\sigma)(t,x)=0,
  %\label{starbc1}\\
  \quad \overline{\sigma}(t,M(x_2,x_3),x_2,x_3)=0.
  %\label{starbc2}
\end{gather*}
\begin{proof}[Outline of proof of Theorem \ref{1.4}]
We rewrite the above initial--boundary value problem over $\Om$
to that over the half space 
$\mathbb R^3_+:=\{y=(y_1,y_2,y_3)\in\mathbb R^3 \, | \, y_1>0\}$ by changing variables
\begin{align*}
 y_1=x_1-M(x_2,x_3), \quad y_2=x_2, \quad y_3=x_3.
\end{align*}
The rewritten problem is given by the system
%\begin{subequations}\label{stare0'}
 \begin{align*}
  {}&
  \D_t\overline{\psi} + (\bm{u}^s+\overline{\bm{\eta}}) \cdot \nabla \overline{\psi} 
  + {\rm{div}} \overline{\bm{\eta}} + \overline{\bm{\eta}}\cdot\nabla w^s 
  \\
  &\qquad =(\partial_{y_1}\overline{\psi})\{(\bm{u}^s+\overline{\bm{\eta}}) \cdot \nabla M\}
   +(\partial_{y_1} \overline{\bm{\eta}})\cdot\nabla M
   +(\partial_{y_1}w^s)\{\overline{\bm{\eta}} \cdot \nabla M\},
  \\
  &\D_t\overline{\bm{\eta}} + (\bm{u}^s+\overline{\bm{\eta}}) \cdot \nabla \overline{\bm{\eta}} 
  +K \nabla \overline{\psi} - \nabla \overline{\sigma} + \overline{\bm{\eta}} \cdot \nabla \bm{u}^s 
  \\
  & \qquad = \{(\bm{u}^s+\overline{\bm{\eta}}) \cdot \nabla M\} \partial_{y_1} \overline{\bm{\eta}} 
  +K(\partial_{y_1}\overline{\psi})\nabla M
  -(\partial_{y_1}\overline{\sigma})\nabla M
  +(\overline{\bm{\eta}}\cdot\nabla M)\partial_{y_1}\bm{u}^s,
  \\
  &\Delta \overline{\sigma}
  -e^{(w^s+\overline{\psi})}+e^{w^s}
  +e^{-\phi^s-\overline{\sigma}}-e^{-\phi^s}
  \\ 
  &\qquad = \sum_{j=2}^3\left[\left(-(\partial_{y_j}M)\partial_{y_1}+\partial_{y_j}\right)
   ((\partial_{y_j}M)\partial_{y_1}\overline{\sigma})
   +(\partial_{y_j}M)\partial^2_{y_1y_j}\overline{\sigma}\right], \quad
  (t,y)\in {\mathbb R}_+\times {\mathbb R}^3_+
 \end{align*}
%\end{subequations}
with the initial and boundary data
\begin{gather*}
  (\overline{\psi},\overline{\bm{\eta}})(0,y)=(\log\rho_0-w^s,\bm{u}_0-\bm{u}^s)(y),
  \quad %\label{staini1'}\\
  \lim_{|y|\rightarrow \infty} (\overline{\psi},\overline{\bm{\eta}},\overline\sigma)(t,y)=0,
  \quad %\label{starbc1'}\\
  \overline{\sigma}(t,0,y_2,y_3)=0.
  %\label{starbc2'}
\end{gather*}
We remark that the left hand sides of the above three equations 
are essentially same as equations (1.13) in \cite{NOS}
and all terms of the right hand side have $\nabla M$.
Therefore, if $\|M\|_5 \ll 1$, by the method of 
the proof of Theorem 1.3 in \cite{NOS} with tiny modifications,
one can show that 
the solution $(\overline{\psi},\overline{\bm{\eta}},\overline{\sigma})$
to the rewritten problem %over the half space $\mathbb R^3_+$
exists globally in time and decays algebraically fast as $t$ tends to infinity.
These facts immediately verify Theorem \ref{1.4}.
%We omit details, since it is straightforward.
\end{proof}

\section{General inequalities}\lb{Appendix1}

\begin{lem}
Let $l=0,1,2,\cdots$ and $\beta \in [0,1]$.
Suppose that $A \in {\cal B}^\infty(B(0,r))$, 
$A(0)=0$, and $\tilde{A} \in {\cal B}^{l+1}(\overline{\Om})$,
where $B(0,r)\subset \mathbb R^n$ denotes a ball
of center $O$ and radius $r \in (0,1]$.
If $f\in L^\infty(\Om) \cap H^l(\Om)$,
$g\in H^l_\beta(\Om)$, and $e^{\beta x_1/2}g \in L^\infty(\Om)$, it holds that
%\begin{subequations}\lb{ineq0}
\begin{gather}
\|fg\|_{l,\beta}
\leq C(\|f\|_{L^\infty}\|g\|_{l,\beta}
+\|f\|_{l}\|e^{\beta x_1/2}g\|_{L^\infty}),
\lb{A1} \\
\|A(f)\|_l \leq C\|f\|_l 
\quad \text{if $\|f\|_{L^\infty} \leq r/2$}.
\lb{A3} 
\end{gather}
If $f,\nabla f\in L^\infty(\Om) \cap H^{l}(\Om)$,
$g\in H^l_\beta(\Om)$, and $e^{\beta x_1/2}g \in L^\infty(\Om)$,
the following inequalities on the commutator $[\nabla^l,\, \cdot\,]$ hold.
\begin{gather}
\|[\nabla^{l+1},f] g \|_{0,\beta} \leq 
C(\|\nabla f\|_{L^\infty}\|g\|_{l,\beta}
+\|\nabla f\|_{l}\|e^{\beta x_1/2}g\|_{L^\infty}),
\lb{A2} \\
\|[\nabla^{l+1},\tilde{A}] g \|_{0,\beta} 
\leq C \left(\sum_{i=1}^{l+1}\|\nabla^{i} \tilde{A}\|_{L^\infty} \right)
\|g\|_{l,\beta}.
\lb{A4} 
\end{gather}
%\end{subequations}
Here $C$ is a positive constant independent of $f$, $g$, and $\beta$. 
\end{lem}
\begin{proof}
Following the proofs of Lemmas 4.8 and 4.9 in \cite{Rac},
we first have \er{A3},
\begin{gather}
\|(\D_x^{\bm{a}}f)(\D_x^{\bm{b}}g)\| 
\leq C(\|f\|_{L^\infty}\|g\|_{l}+\|f\|_{l}\|g\|_{L^\infty})
\quad  \text{if $|\bm{a}|+|\bm{b}| \leq l$},
\lb{A7}
\\
\|fg\|_{l} \leq C(\|f\|_{L^\infty}\|g\|_{l}+\|f\|_{l}\|g\|_{L^\infty}).
\lb{A6}
\end{gather}
Let us show \er{A1} by using \er{A6}. It is easy to see that
\[
\|e^{\beta x_1/2}(fg)\|_l \leq C(\|f\|_{L^\infty}\|e^{\beta x_1/2}g\|_{l}
+\|f\|_{l}\|e^{\beta x_1/2}g\|_{L^\infty}).
\]
This together with the equivalence of norms \er{eqiv0} leads to \er{A1}.

For the commutator $[\nabla^{l+1},\, \cdot\,]$, one can obtain easily \er{A4}.
Therefore, we prove only inequality \er{A2}. 
For any $\bm{a}$ with $|\bm{a}|=l+1$, 
\[
e^{\beta x_1/2}[\D^{\bm{a}},f] g =
\sum_{\bm{b} \leq \bm{a}, |\bm{b}|\neq 0} 
C[\bm{a},\bm{b}] (\D^{\bm{b}} f) e^{\beta x_1/2}(\D^{\bm{a}-\bm{b}} g).
\]
Furthermore, it is shown by induction that
\[
e^{\beta x_1/2}(\D^{\bm{a}-\bm{b}} g) =
\sum_{\bm{c} \leq \bm{a}-\bm{b}} C[\bm{c},\beta]\D^{\bm{c}}(e^{\beta x_1/2}g).
\]
Eventually, $e^{\beta x_1/2}[\D^{\bm{a}},f] g$ can be represented 
by a linear combination of terms
\[
(\D^{\bm{b}_1} \D^{\bm{b}_2} f) \D^{\bm{c}}(e^{\beta x_1/2}g),
\]
where $|\bm{b}_2|=1$, $|\bm{b}_1|\leq l$, $|\bm{c}| \leq l$, and $|\bm{b}_1|+|\bm{c}|\leq l$.
Then, applying \er{A7} to these terms, we conclude \er{A2}.
%\begin{align*}
%\|(\D^{\bm{b}_1} \D^{\bm{b}_2} f) \D^{\bm{c}}(e^{\beta x_1/2}g)\|
%&\leq C\{|\nabla f|_{\infty}\|e^{\beta x_1/2}g\|_{l-1}+\|f\|_{l}|e^{\beta x_1/2}g|_{\infty}\}
%\\
%&\leq C\{|\nabla f|_{\infty}\|g\|_{l-1,\beta}+\|f\|_{l}|g|_{\infty,\beta}\}.
%%%\quad (\because \text{\er{equiv1}})
%\end{align*}
%The proof is complete.
\end{proof}

\begin{lem}\lb{Gell1}
Let $l=0,1,2,\ldots$.
Consider the boundary value problem
\begin{gather*}
-\Delta\sigma +\sigma = \psi \quad \text{in $\Om$},
\\
\sigma=0 \quad \text{on $\D \Om$}.
\end{gather*}
If $\psi \in H^l(\Om)$, 
then this problem has a unique solution $\sigma \in H^{l+2}(\Om)$.
Moreover, there exists a positive constant $C$ independent of $\psi$ and $\sigma$ such that
\begin{equation*}
 \|\sigma\|_{l+2} \leq C\|\psi\|_{l}.
\end{equation*}
\end{lem}
\begin{proof}
This can be shown in much the same way as Theorems 4 and 5 in Section 6.3 in \cite{Ev}.
\end{proof}
\end{appendix}

\end{document}